\documentclass[times]{article}
\usepackage[utf8]{inputenc}
\usepackage{geometry}
\usepackage[titletoc,toc,title]{appendix}
\usepackage{graphicx}
\usepackage{amssymb}
\usepackage{epstopdf}
\usepackage{mathrsfs}
\usepackage{amsmath}
\usepackage{amsthm}
\usepackage{csquotes}
\usepackage{hyperref}
\usepackage{authblk}
\usepackage{xcolor}

\begin{document}
\newtheorem{theoreme}{Theorem}[section]
\newtheorem{ex}{Example}[section]
\newtheorem{definition}{Definition}[section]
\newtheorem{lemme}{Lemma}[section]
\newtheorem{remarque}{Remark}[section]
\newtheorem{exemple}{Example}[section]
\newtheorem{proposition}{Proposition}[section]
\newtheorem{corolaire}{Corollary}[section]
\newtheorem{hyp}{Hypothesis}[section]
\newtheorem*{rec}{Recurrence Hypothesis}
\newtheorem*{prop}{Proposition}
\newtheorem*{theo}{Theorem}
\newtheorem*{Conj}{Conjecture}
\newcommand\Bound{\partial\overline{X}}
\newcommand\bdf{b_\xi}
\newcommand\N{\mathbb{N}}
\newcommand\Euc{\mathbb{R}^d}
\newcommand\COIn{\tilde{\chi}}
\newcommand\Comp{\overline{X}}
\newcommand\Dir{\varphi}
\newcommand\Lag{\mathcal{L}_0}
\newcommand\ins{\gamma_{uns}}
\newcommand\ext{\mathcal{W}_{0}}
\newcommand\brut{A_7}
\newcommand\inter{\mathcal{W}_2}
\newcommand\close{\mathcal{W}_1}
\newcommand\tins{N_{uns}}
\newcommand\din{\delta_{in}}
\newcommand\inco{\mathcal{DE}_-}
\newcommand\se{\epsilon_{sec}}
\newcommand\Nlag{N_{lag}}
\newcommand\cmin{c_{min}}
\newcommand\ind{J_a^{h,r}}
\newcommand\sor{\mathcal{T}}
\newcommand\secur{\varepsilon_0}
\newcommand\diam{\varepsilon_6}
\newcommand\tim{T_0}
\newcommand\zone{\overline{a}}
\newcommand\exit{N_\epsilon}
\newcommand\wait{N_1}
\newcommand\cli{\varepsilon_7}
\newcommand\nuag{\chi_{7}}
\newcommand\rain{\Pi_7}
\newcommand\sui{A_1^{\leq N}}
\newcommand\symp{\mathcal{C}_0}
\newcommand\zeit{t_2}
\newcommand\petit{\varepsilon_2}
\newcommand\temps{t_5}
\newcommand\set{\mathcal{W}_3}
\newcommand\taille{\varepsilon_2}
\newcommand\bel{>}
\newcommand\R{\mathbb{R}}
\newcommand\C{\mathbb{C}}
\newcommand\Sp{\mathbb{S}^{d-1}}
\newcommand\spt{\mathrm{supp}}
\makeatletter
\renewcommand\theequation{\thesection.\arabic{equation}}
\@addtoreset{equation}{section}
\makeatother
\title{Semiclassical limits of distorted plane waves in chaotic scattering without a pressure condition}
\author{Maxime Ingremeau}
\affil{Laboratoire J.A.Dieudonn\'e, UMR CNRS-UNS 7351, Universit\'e C\^ote
  d'Azur, 06000 Nice, France\thanks{maxime.ingremeau@univ-cotedazur.fr}}

\date{}

\maketitle
\begin{abstract}
In this paper, we study the semi-classical behavior of distorted plane waves, on manifolds that are Euclidean near infinity or hyperbolic near infinity, and of non-positive curvature. Assuming that there is a strip without resonances below the real axis, we show that distorted plane waves are bounded in $L^2_{loc}$ independently of $h$, that they admit a unique semiclassical measure, and we prove bounds on their $L^p_{loc}$ norms.
\end{abstract}

\section{Introduction}

Consider a Riemannian manifold $(X,g)$ of dimension $d\geq 2$ which is \emph{Euclidean near infinity}, that is to say, such that there exists a compact set $X_0\subset X$ and $R_0\bel 0$ such that $(X\backslash X_0,g)$ and $(\mathbb{R}^d\backslash B(0,R_0),g_{eucl})$ are isometric\footnote{More generally, we could consider manifolds with several Euclidean ends, i.e., such that $(X\backslash X_0,g)$ is isometric to $((\mathbb{R}^d\backslash B(0,R_0))^k,g_{eucl})$ for some $k\in \N$. This would add no conceptual change, but would make notations heavier, so we will only write our statements on manifolds with a single infinite end.}. The \emph{distorted plane waves} on $X$ are a family of functions $E_h(x;\xi)$ with parameters $\xi\in\mathbb{S}^{d-1}$ (the direction of propagation of the incoming wave) and $h$ (a semiclassical parameter corresponding to the inverse of the square root of the energy) such that
\begin{equation}\label{eigen}
(-h^2\Delta-1) E_h(x;\xi)=0, 
\end{equation}
and which can be put in the form
\begin{equation}\label{jeanne}
E_h(x;\xi)= E_h^0(x; \xi) + E_{out}.
\end{equation}
Here, 
$$E_h^0(x;\xi) = (1-\chi_0) e^{\frac{i}{h} x\cdot\xi},$$
where $\chi_0\in C_c^\infty$ is such that $\chi_0\equiv 1$ on $X_0$, and $E_{out}(x;\xi,h)$ is \emph{outgoing} in the sense that it satisfies the \emph{Sommerfeld radiation condition}, where $|x|$ is the distance to any fixed point in $X$:
\begin{equation}\label{Sommerfeld}
\lim \limits_{|x|\rightarrow \infty} |x|^{(d-1)/2} \Big{(} \frac{\partial}{\partial |x|} - \frac{i }{h}\Big{)} E_{out} = 0.
\end{equation}

It can be shown (cf. \cite[\S 2]{Mel95} or \cite[\S 4]{DZ19}) that there is a unique function $E_h(\cdot;\omega)$ such that (\ref{eigen}) is satisfied and which can be put in the form (\ref{jeanne}). It does not depend on the choice of the function $\chi_0\in C_c^\infty(X)$, as long as $\chi_0\equiv 1$ on $X_0$.

Actually, the term $E_{out}$ can be given an explicit expression in terms of the \emph{outgoing resolvent}, that is to say, in terms of the family of operators $R_+(z;h) = (-h^2\Delta-z)^{-1}$, which is well defined for $\Im(z)>0$ as an operator from $L^2(X)$ to itself.

It is well-known (see \cite[\S 4 and \S5]{DZ19}) that, if $\chi \in C_c^\infty(X)$, then for any $h>0$, $z \mapsto \chi R_+(z;h)\chi$ can be extended to $\C\backslash (-\infty,0]$ as a meromorphic function. Its poles, which are independent of the choice of $\chi$, are called the \emph{resonances} of $-h^2\Delta$. Since there are no resonances on $[0,\infty)$, $R_+(1;h)$ is well defined as an operator $L^2_{comp}\longrightarrow L^2_{loc}$.

 With $\chi_0$ as in (\ref{jeanne}), we set
\begin{equation*}
F_h(\cdot;\xi)= [h^2\Delta,\chi_0] e^{\frac{i}{h} x\cdot\xi},
\end{equation*}
which is compactly supported, and satisfies $\|F_h\|_{L^2} =O(h)$.

We then have
\begin{equation}\label{eq: def Eout}
E_{out}(\cdot;\xi,h) := R_+(1;h)F_h(\cdot;\xi).
\end{equation}

In this paper, we will be interested in the behavior of $E_h(\cdot;\xi)$ in the semiclassical limit $h\rightarrow 0$. The first question we would like to address is whether $E_h$ is bounded in $L^2_{loc}$ uniformly with respect to $h$. 
More generally, we will be interested in the semiclassical limits of $E_h$, and in the behavior of the $L^p_{loc}$ norms of $E_h(\cdot;\xi)$ as $h\rightarrow 0$.

In \cite{Ing} and \cite{Ing2}, the author answered these questions under some assumptions on the dynamics of the geodesic flow. Let us denote by $p$ the classical Hamiltonian $p: T^*X \ni (x,\xi)\mapsto \|\xi\|_x^2\in \mathbb{R}$.

For each $t\in \mathbb{R}$, we denote by $\Phi^t:T^*X\longrightarrow T^*X$
the geodesic flow generated by $p$ at time $t$. We will denote by the same letter its restriction $\Phi^t : S^*X \longrightarrow S^*X$ to the energy layer $p(x,\xi)=1$.

The \emph{trapped set} is defined as
\begin{equation}\label{defcapte}
K:=\{(x,\xi)\in S^*X; \Phi^t(x,\xi) \text{ remains in a bounded set for all } t\in \mathbb{R}\}.
\end{equation}

One of the main results of \cite{Ing} was the following. Suppose that the trapped set is \emph{hyperbolic}, and that \emph{the topological pressure associated to half the unstable Jacobian is negative}: $\mathcal{P}(1/2)<0$. (see section \ref{averse} for the definition of a hyperbolic set, and of the topological pressure). Then\footnote{Actually, in \cite{Ing}, we also make a transversality assumption on the direction $\xi$, which is always satisfied if $(X,g)$ has non-positive curvature.} $E_h$ is uniformly bounded in $L^2_{loc}(\cdot;\xi)$, and it has a unique semiclassical measure.
In \cite{Ing2}, under the additional assumption that $(X,g)$ has  non-positive curvature, it was shown that $E_h(\cdot;\xi)$ is uniformly bounded in $L^\infty_{loc}$.

The aim of this paper is to extend some of the results of \cite{Ing} and \cite{Ing2} in the case where no assumption is made on the topological pressure associated to half the unstable Jacobian. Instead, we will make the weaker assumption that there is a \emph{resonance-free strip below the real axis}.

\subsubsection*{Resonance-free strip}
In the sequel, we will suppose that there exists $\varepsilon_0, h_0, C_0, >0$, such that for all $0<h<h_0$, $-h^2\Delta$ has no resonances in
\begin{equation}\label{eq: def D}
\mathcal{D}_h:=\Big{\{ }z\in \C ; \Re z \in [1-\varepsilon_0, 1+\varepsilon_0] \text{ and } \Im z \geq - C_0 h\Big{\}}.
\end{equation}
Furthermore, we suppose that there exists $\alpha> 0$ such that the following holds. For any $\chi \in C_c^\infty(X)$, there exists $C_\chi>0$ such that for all $z\in \mathcal{D}_h$,
\begin{equation}\label{eq:aprioribound}
\|\chi R_+(z;h)\chi \|_{L^2\mapsto L^2} \leq C_\chi h^{-\alpha}.
\end{equation}

It was shown in \cite{NZ}, \cite{nonnenmacher2009semiclassical} that (\ref{eq:aprioribound}) holds when the topological pressure $\mathcal{P}(1/2)$ is strictly negative. In \cite{bourgain2016spectral}, (\ref{eq:aprioribound}) was shown to hold on all convex co-compact surfaces, even when the condition $\mathcal{P}(1/2)<0$ is not satisfied. By gluing resolvent estimates thanks to the methods of \cite{DV}, we can modify a convex co-compact surface near infinity, by replacing the hyperbolic funnels by Euclidean ends, so that (\ref{eq:aprioribound}) still holds. Hence, there are some examples of Euclidean near infinity manifolds such that (\ref{eq:aprioribound}) holds, but $\mathcal{P}(1/2)\geq 0$. Actually, it was conjectured in \cite[Conjecture 3, \S 3.2]{zworski2017mathematical} that (\ref{eq:aprioribound}) holds on any Euclidean near infinity manifold with a compact hyperbolic trapped set.

\begin{theoreme}\label{th : L2}
Let $(X,g)$ be a Riemannian manifold which is Euclidean near infinity. We suppose  that $(X,g)$ has nonpositive sectional curvature, that the trapped set is hyperbolic (Hypothesis \ref{sieste}), and that (\ref{eq:aprioribound}) is satisfied.
Let $\xi\in\mathbb{S}^{d-1}$ and $\chi\in C_c^\infty(X)$. Then there exists $C_{\xi,\chi}\bel 0$ such that, for any $h\bel 0$, we have
\begin{equation*}
\|\chi E_h(\cdot,\xi)\|_{L^2}\leq C_{\xi,\chi}.
\end{equation*}
\end{theoreme}

We will actually give more precise results about the $L^p_{loc}$ norms and the semi-classical measure of $E_h$ in section \ref{Sec: Statements}.

\subsubsection*{More general framework}
The study of distorted plane waves in the presence of a hyperbolic trapped set is not restricted to the framework of Euclidean near infinity manifolds, though this is the situation we consider in this introduction. Our results will also hold in the case of hyperbolic near infinity manifolds (see the next section for the definition of distorted plane waves in this setting). In particular the analogue of Theorem \ref{th : L2} holds for Eisenstein series on all convex co-compact hyperbolic surfaces, thanks to the results of \cite{bourgain2016spectral}, which guarantee that (\ref{eq: def D}) and (\ref{eq:aprioribound}) hold.

Our results would also hold if (\ref{eq: def D}) and (\ref{eq:aprioribound}) were replaced by the weaker assumption that there are no resonances and that the resolvent is polynomially bounded in the smaller region
$$\mathcal{D}'_h:=\Big{\{ }z\in \C ; \Re z \in [1-\varepsilon_0, 1+\varepsilon_0] \text{ and } \Im z \geq - C_0 \frac{h}{|\log h|^\beta}\Big{\}}$$
for some $\beta\geq 0$. However, since we know no example of such a resonance free strip with $\beta>0$, so we will only work with the assumptions (\ref{eq: def D}) and (\ref{eq:aprioribound}).

Note that, by the methods of \cite{vodev2014semi} (see also \cite{datchev2012extending}), one can show that this assumption amounts to having $\|\chi R_+(1;h)\chi \|_{L^2\mapsto L^2} \leq C_\chi h^{-1} |\log h|^{\beta}$.

\subsubsection*{Relation to other works}
The study of the high frequency behaviour of eigenfunctions of the Laplacian, and of their semiclassical measures, in the case where the associated classical dynamics has a chaotic behaviour, has a long history. It goes back to the classical works \cite{Shn},\cite{Zel} and \cite{CdV} dealing with Quantum Ergodicity on compact manifolds.

Analogous results on manifolds of infinite volume are much more recent. Although distorted plane waves are a natural family of eigenfunctions, they may not be uniformly bounded in $L_{loc}^2$, so that it may not be possible to define their semiclassical measure.

In \cite{DG}, the authors studied the semiclassical measures associated to distorted plane waves in a very general framework, with very mild assumptions on the classical dynamics. The counterpart of this generality is that the authors have to average on directions $\xi$ and on an energy interval of size $h$ to be able to define the semiclassical measure of distorted plane waves. Their result can be seen as a form of Quantum Ergodicity result on non-compact manifolds, although no “ergodicity” assumption is made.

In \cite{GN}, the authors considered the case where $X=\Gamma\backslash \mathbb{H}^{d}$ is a manifold of infinite volume, with sectional curvature constant equal to $-1$ (convex co-compact hyperbolic manifold), and with the assumption that the Hausdorff dimension of the limit set of $\Gamma$ is smaller than $(d-1)/2$. In this setting, distorted plane waves are often called \emph{Eisenstein series}. The authors prove that there is a unique semiclassical measure for the Eisenstein series with a given incoming direction, and they give a very explicit formula for it. This result can hence be seen as a Quantum Unique Ergodicity result in infinite volume.

The results of \cite{GN} were extended to the case of variable curvature in \cite{Ing}, \cite{Ing2}, under the assumption that the topological pressure of half the unstable Jacobian is negative: $\mathcal{P}(1/2)<0$, which naturally generalizes the assumption that the Hausdorff dimension of the limit set of $\Gamma$ is smaller than $(d-1)/2$.

Showing resonance gaps without the assumption $\mathcal{P}(1/2)<0$ is a very delicate issue (see \cite{Nonnen} and \cite{zworski2017mathematical} for a review of the known results and conjectures). Actually, it is not even known that for a general hyperbolic trapping, the resolvent is polynomially bounded on the real axis. The main examples where a resonance-free strip with polynomial bounds on the resolvent is known are convex co-compact hyperbolic surfaces (\cite{bourgain2016spectral}) as well as some families of convex co-compact hyperbolic manifolds of higher dimension (\cite{dyatlov2016spectral}).

In this paper, we will also study the behavior of the $L^p_{loc}$ norms of $E_h$ as $h$ goes to zero. To this end, we will use a method introduced in \cite{hezari2016quantum}, which consists in showing $L^2$ bounds on $E_h$ restricted to balls whose radius depend on $h$.

\subsubsection*{Structure of the paper}
In section \ref{sec2}, we will recall the definition of hyperbolicity and topological pressure, and state our results on distorted plane waves. In particular, we will describe their semiclassical measure, and show bounds on their $L_{loc}^p$ norms. In section \ref{sec:tools}, we shall recall a few facts of classical dynamics which were proven in \cite{Ing} and \cite{Ing2}. We will give the proof of our results in section \ref{sec:proof}. Finally, we shall recall a few facts of semiclassical analysis in appendix \ref{pivoine}.

\subsubsection*{Acknowledgements}

The author would like to thank Stéphane Nonnenmacher and Colin Guillarmou for useful discussion. He is also very thankful to the anonymous referees for their valuable remarks.

The author was partially supported by the Agence Nationale de la Recherche project GeRaSic (ANR-13-BS01-0007-01).
\section{Assumptions and statement of the results}\label{sec2}
Before recalling the definitions of hyperbolicity and of topological pressure, let us recall how distorted plane waves can be constructed on manifolds that are hyperbolic near infinity.
\subsection{The case of hyperbolic  near infinity manifolds}
Our results do not apply only in the case of Euclidean near infinity manifolds, but also in the case of hyperbolic near infinity manifolds. We shall recall here the definition of distorted plane waves on hyperbolic near infinity manifolds. In the framework of convex co-compact hyperbolic manifolds, distorted plane waves are often referred to as \emph{Eisenstein series}.

\begin{definition}\label{def:HypInf}
 We say that $X$ is \emph{hyperbolic near infinity} if it fulfills
the following assumptions.
\begin{enumerate}
\item There exists a compactification $\Comp$ of $X$, that is, a compact
manifold with boundaries $\Comp$ such that $X$ is diffeomorphic to the
interior of $\Comp$. The boundary $\Bound$ is called the boundary at
infinity.

\item There exists a boundary defining function $b$ on $X$, that is, a
smooth function $b : \Comp \longrightarrow [0,\infty)$ such that $b>0$ on
$X$, and $b$ vanishes to first order on $\Bound$.

\item In a collar
neighborhood of $\partial \overline{X}$, the metric $g$ has sectional
curvature $-1$ and can be put in the form
\begin{equation*}g=\frac{\mathrm{d}b^2+h(b)}{b^2},
\end{equation*}
 where $h(b)$ is a smooth 1-parameter family of metrics on $\partial\overline{X}$ for $b\in [0,\epsilon)$.
\end{enumerate}
\end{definition}

\subsubsection*{Construction of $E_{out}$}
Let us fix a $\xi \in \partial \overline{X}$. Since $X$ is hyperbolic near infinity, there exists a neighborhood $\mathcal{V}_\xi$ of $\xi$
in $\overline{X}$ and an isometric diffeomorphism $\psi_\xi$ from
$\mathcal{V}_\xi\cap X$ into a neighborhood $V_{q_0,\delta}$ of the
north pole $q_0$ in the unit ball $\mathbb{B}:=\{q\in \mathbb{R}^d;
|q|<1\}$ equipped with the metric $g_0$:
\begin{equation*}V_{q_0,\delta} :=\{q\in \mathbb{B}; |q-q_0|<\delta\},~~~~ g_0=\frac{4
\mathrm{d}q^2}{(1-|q|^2)^2},
\end{equation*}
where $\psi_\xi(\xi)=q_0$, and $|\cdot|$ denotes the Euclidean length. We
shall choose the boundary defining function on the ball $\mathbb{B}$ to be
\begin{equation}\label{clovis}
b_0=2\frac{1-|q|}{1+|q|},
\end{equation}
and the induced metric $b_0^2 g_0|_{\mathbb{S}^d}$ on $\mathbb{S}^d=
\partial \mathbb{B}$ is the usual one with curvature $+1$. The function
$b_\xi:= b_0\circ \psi_\xi^{-1}$ can be viewed locally as a boundary defining function on $\overline{X}                                                                                                                                                                                                                                                                                                                                                                                                                                                                                                                                                                                                                                                                                                                                                                                                                                                                                                                                                                                                                                                                                                                                                                                                                                                                                                                                                                                                                                                                                                                                                                                                                                                                                                                                                                                                                                                                                                                                                                                                                                                                                                                                                                                                                                                                                                                                                                                                                                                                                                                                                                                                                                                                                                                                                                                                                                                                                                                                                                                                                                                                                                                                                                                                                                                              $.

For each $p\in \mathbb{S}^d$, we define the Busemann function on $\mathbb{B}$
\begin{equation*}\phi_p^{\mathbb{B}}(q)=\log \Big{(}\frac{1-|q|^2}{|q-p|^2}\Big{)}.
\end{equation*}

There exists an $\epsilon>0$ such that the set
\begin{equation*}U_\xi := \{x\in \overline{X}; d_{\overline{g}(x,\xi)<\epsilon}\}
\end{equation*}
lies inside $\mathcal{V}_{\xi}$, where $\overline{g}= \bdf^2g$ is the
compactified
metric.
We define the function
\begin{equation*}\phi_\xi(x):=
\phi^{\mathbb{B}}_{q_0}
\big{(}\psi_{\xi}(x)\big{)}, ~~\text{    for   } x\in U_{\xi}, ~~ 0 \text{ otherwise}.
\end{equation*}

Let $\chi_0:\overline{X}\longrightarrow [0,1]$ be a smooth function which
vanishes outside of $U_\xi$, which is equal to one in a neighborhood
of $\xi$.

The incoming wave is then defined as
\begin{equation*}E_h^0(x;\xi) :=\chi_0(x) e^{((d-1)/2+i/h)\phi_\xi(x)}~~
\text{ if } x\in U_\xi ,~~~~0 \text{ otherwise.} 
\end{equation*}

$E_h^0$ is then a Lagrangian state associated to the Lagrangian manifold 
\begin{equation}\label{eq:DefLag}
\Lambda_\xi= \{ (x,\partial_x \phi_\xi(x)),  x\in U_\xi\}.
\end{equation}

\subsubsection*{Construction of $E_h^1$}
We set $E_{out}:= R_+(1;h) F_h$, where $R_+(1;h)$ is the outgoing resolvent\footnote{which is still supposed to satisfy (\ref{eq:aprioribound}).}  $$\left(-h^2\Delta -\frac{(d-1)^2}{4} h^2 - 1 -i0\right)^{-1},$$ and $F_h:= [h^2\Delta, \chi_0] e^{((d-1)/2+i/h)\phi_\xi(x)}$. 

We then define 
$$E_h := E_h^0 + E_{out}.$$

We refer the reader to \cite[\S 7]{DG} for more details on this construction, and other equivalent definitions of distorted plane waves in this context, showing that our definition is intrinsic.

\subsection{Assumptions on the classical dynamics} \label{averse}
Let $(X,g)$ be a manifold which is Euclidean near infinity, or hyperbolic near infinity.

In the sequel, we will always assume that $(X,g)$ has non-positive sectional curvature. Since the curvature is constant outside of a compact set, we may define
\begin{equation}\label{eq: def b0}
-b_0 \text{ is the minimal value taken by the sectional curvature on } X.
\end{equation}

Let us describe more precisely the hyperbolicity assumption we make. It will imply that $b_0>0$.
\subsubsection*{Hyperbolicity}
For $\rho\in S^*X$, we will say that $\rho\in \Gamma^\pm$ if $\{\Phi^t(\rho), \pm t\leq 0\}$
is a bounded subset of $S^*X$; that is to say, $\rho$ does not “go to
infinity”, respectively in the past or
in the future. The sets $\Gamma^\pm$ are called respectively the
\textit{outgoing} and \textit{incoming} tails.

The \textit{trapped set} is defined as
\begin{equation*}K:=\Gamma^+\cap \Gamma^-.
\end{equation*}
It is a flow invariant set, and it is compact by the geodesic convexity assumption.

\begin{hyp}[Hyperbolicity of the trapped set] \label{sieste}
We assume that $K$ is non-empty, and is a hyperbolic set for the flow
$\Phi^t$. That is to say,
there exists an adapted metric $g_{ad}$ on a neighborhood of $K$ included in
$S^*X$, and $\lambda>0$, such that the following holds. For each
$\rho\in K$, there is a decomposition \begin{equation*}T_\rho(S^*X)=\mathbb{R}\frac{\partial \big{(}\Phi^t(\rho)\big{)}}{\partial t} \oplus E_\rho^+\oplus
E_\rho^-
\end{equation*} such that
\begin{equation*}\|d\Phi_\rho^t(v)\|_{g_{ad}}\leq  e^{-\lambda|t|}\|v\|_{g_{ad}}
\text{   for all } v\in E_\rho^\mp, \pm t\geq 0.
\end{equation*}
\end{hyp}

The spaces $E_\rho^\pm$ are respectively called the \emph{unstable} and \emph{stable} spaces at $\rho$. The \emph{weakly unstable} space at $\rho$ is defined by $E_\rho^{0+}:=\mathbb{R}\frac{\partial \big{(}\Phi^t(\rho)\big{)}}{\partial t} \oplus E_\rho^+$.

We may extend $g_{ad}$ to a metric
on $S^*X$, so that outside of the interaction region $X_0$, it coincides with the restriction of the metric on $T^*X$ induced from the Riemannian metric on $X$. From now on, we will denote by
\begin{equation*}d_{ad} \text{ the Riemannian
distance associated to the metric } g_{ad} \text{ on } S^*X.
\end{equation*}

\subsubsection*{Topological pressure}
 We shall say
that a set $\mathcal{S}\subset K$ is $(\epsilon,t)$-separated if for
$\rho_1, \rho_2\in \mathcal{S}$, $\rho_1 \neq \rho_2$, we have
$d_{ad}(\Phi^{t'}(\rho_1),\Phi^{t'}(\rho_2))>\epsilon$ for some $0\leq t \leq
t'$. (Such a set is necessarily finite.)

The metric $g_{ad}$ induces a volume form $\Omega$ on any $d$-dimensional
subspace of $T(T^*X)$. Using this volume form, we will define
the unstable Jacobian on $K$. For any $\rho\in K$, the determinant map
\begin{equation*}\Lambda^d d\Phi^t(\rho)|_{E_\rho^{+0}} : \Lambda^d E_\rho^{+0}
\longrightarrow \Lambda^d E_{\Phi^t(\rho)}^{+0}
\end{equation*}
can be identified with the real number
\begin{equation*}\det\big{(} \mathrm{d}\Phi^t(\rho)|_{E_\rho^{+0}}\big{)} :=
\frac{\Omega_{\Phi^t(\rho)}\big{(}\mathrm{d}\Phi^tv_1 \wedge
\mathrm{d}\Phi^tv_2\wedge...\wedge \mathrm{d}\Phi^tv_d\big{)}}{\Omega_\rho(v_1\wedge v_2
\wedge... \wedge v_d)},
\end{equation*}
where $(v_1,...,v_d)$ can be any basis of $E_\rho^{+0}$. This number
defines the unstable Jacobian:
\begin{equation}\label{defJaco}
\exp \lambda^+_t(\rho) := \det\big{(}
\mathrm{d}\Phi^t(\rho)|_{E_\rho^{+0}}\big{)}.
\end{equation}
From there, we take
\begin{equation*}Z_t(\epsilon,s):= \sup \limits_{\mathcal{S}} \sum_{\rho \in \mathcal{S}}
\exp(-s\lambda_t^+(\rho)),
\end{equation*}
where the supremum is taken over all $(\epsilon,t)$-separated sets. The
pressure is then defined as
\begin{equation}\label{eq: defpres}
\mathcal{P}(s):= \lim \limits_{\epsilon \rightarrow 0} \limsup \limits_{t
\rightarrow \infty} \frac{1}{t} \log  Z_t(\epsilon,s) .
\end{equation}
This quantity is actually independent of the volume form $\Omega$ and of the metric chosen: after
taking logarithms, a change in $\Omega$ or in the metric will produce a term $O(1)/t$,
which is not relevant in the $t\rightarrow \infty$ limit.

One of the main assumptions made in \cite{Ing} and \cite{Ing2} was that $\mathcal{P}(1/2)<0$. Here, we will use instead the weaker fact, proven in  \cite[Theorem 5.6]{bowen1975ergodic} that
\begin{equation}\label{eq : pression en 1}
\mathcal{P}(1)<0.
\end{equation}

Inequality (\ref{eq : pression en 1}) follows from \cite[Theorem 5.6]{bowen1975ergodic}, and from the fact that $K$ is a basic hyperbolic set, but not an attractor.

In the special case where $X = \Gamma\backslash \mathbb{H}$ is a convex co-compact surface (so that $\Gamma$ is a Fuchsian group of the second kind), if we denote by $\delta(\Gamma)$ the exponent of convergence of $\Gamma$, we have (as recalled in \cite[Appendix B]{DG}) $\mathcal{P}(1)= \delta - 1$, so that  condition (\ref{eq : pression en 1}) corresponds to $\delta(\Gamma)<1$, while the condition $\mathcal{P}(1/2)$ would correspond to $\delta<1/2$. It was shown in \cite{beardon1971inequalities} that the exponent of convergence of a Fuchsian group of the second kind always satisfies $\delta<1$, so that (\ref{eq : pression en 1}) is always satisfied.

To take advantage of this fact, we will introduce in section \ref{sec: other press} another definition of the topological pressure, which was introduced in \cite{NZ}. We refer the reader to this paper for the proof that the two definitions are equivalent.

\subsection{Statement of the results}\label{Sec: Statements}

Our main result is the following.

\begin{theoreme} \label{blacksabbath4}
Let $(X,g)$ be a Riemannian manifold which is Euclidean or hyperbolic near infinity. We suppose  that $(X,g)$ has nonpositive sectional curvature, that the trapped set is hyperbolic (Hypothesis \ref{sieste}), and that (\ref{eq:aprioribound}) is satisfied.

Let $\mathcal{K}\subset X$ be a compact set. There exists $\varepsilon_{\mathcal{K}}>0$ such that, if $\chi\in C_c^\infty(X)$ has a support of diameter smaller than $\varepsilon_{\mathcal{K}}$ included in $\mathcal{K}$, the following holds.

There exists an infinite set $\tilde{\mathcal{B}}^\chi$ and a function $\tilde{n}: \tilde{\mathcal{B}}^\chi\rightarrow \mathbb{N}$ such that the number of elements in $\{\tilde{\beta}\in\tilde{\mathcal{B}}^\chi; \tilde{n}(\tilde{\beta})\leq N\}$ grows at most exponentially with $N$, and such that the following holds.

For any $0<M< \frac{1}{2\sqrt{b_0}}$, for any $\varepsilon>0$, we have

\begin{equation}\label{qad2} \chi E_h(x) = \sum_{\substack{\tilde{\beta}\in \tilde{\mathcal{B}}^\chi\\
                \tilde{n}(\tilde{\beta})\leq M|\log h|}} e^{i \varphi_{\tilde{\beta}}(x)/h}
a_{\tilde{\beta},\chi}(x;h) +  O_{L^2}\Big{(}h^{\frac{M}{2}|\mathcal{P}(1)|-\varepsilon}\Big{)},
\end{equation}
where $a_{\tilde{\beta},\chi}\in S^{comp}(X)$ is a classical symbol in the sense of Definition \ref{defsymbclassique}, each $\varphi_{\tilde{\beta}}$ is a smooth real-valued function defined in a neighborhood of the support of
$a_{\tilde{\beta}}$, and the remainder depends on $\varepsilon$.

Furthermore, for any $\ell,n\in \mathbb{N}$
\begin{equation}\label{sheriff3}
\sum_{\substack{\tilde{\beta}\in \tilde{\mathcal{B}}^\chi\\ \tilde{n}(\tilde{\beta})=n}} \|a_{\tilde{\beta},\chi}\|^2_{C^\ell} \leq C_{\varepsilon,\ell} n^{2\ell}
e^{n(\mathcal{P}(1)+\varepsilon)},
\end{equation}
and there exists a constant $C_\chi>0$ such that for all $\tilde{\beta}\neq  \tilde{\beta}'\in \tilde{\mathcal{B}}^\chi$, with $\tilde{n}(\tilde{\beta})\leq M |\log h|$, we have
\begin{equation}\label{hurry8}
|\partial \varphi_{\tilde{\beta}} (x)- \partial \varphi_{\tilde{\beta'}}(x)|\geq C_\chi  h^{\sqrt{b_0}M}. 
\end{equation}
\end{theoreme}

Actually, our proof shows that we may obtain a smaller remainder, of the order of $h^N$ for any $N$, by taking more terms into account in the sum (i.e. replacing $M$ by some larger constant $M(N)$).

\begin{proof}[Proof that Theorem \ref{blacksabbath4} implies Theorem \ref{th : L2}]

By a partition of unity, it suffices to prove the statement for $\chi$ with a small enough support so that Theorem \ref{blacksabbath4} applies.
For such a $\chi$, we may use Theorem \ref{blacksabbath4} for $M= M^\varepsilon = \frac{1}{2\sqrt{b_0}}- \varepsilon$, and denote by $\mathrm{R}_h$ the remainder in (\ref{qad2}). We obtain
\[
\begin{aligned}
\|\chi (E_h-\mathrm{R}_h)\|_{L^2}^2 &= \Big{\langle} \sum_{\substack{\tilde{\beta}\in \tilde{\mathcal{B}}^\chi\\
                \tilde{n}(\tilde{\beta})\leq M^{\varepsilon}|\log h|}} e^{i \varphi_{\tilde{\beta}}(x)/h}
a_{\tilde{\beta},\chi}(x;h) ,\sum_{\substack{\tilde{\beta}\in \tilde{\mathcal{B}}^\chi\\
                \tilde{n}(\tilde{\beta})\leq M^{\varepsilon}|\log h|}} e^{i \varphi_{\tilde{\beta}}(x)/h}
a_{\tilde{\beta},\chi}(x;h)  \Big{\rangle}_{L^2}\\
&=\sum_{\substack{\tilde{\beta} \in \tilde{\mathcal{B}}^\chi\\
                \tilde{n}(\tilde{\beta})\leq M^{\varepsilon}|\log h|}} \|a_{\tilde{\beta},\chi}(x;h) \|_{L^2}^2 + \sum_{\substack{\tilde{\beta}\neq \tilde{\beta}'  \in \tilde{\mathcal{B}}^\chi\\
                \tilde{n}(\tilde{\beta}), \tilde{n}(\tilde{\beta}')\leq M^{\varepsilon}|\log h|}}  \langle  e^{i \varphi_{\tilde{\beta}}(x)/h}
a_{\tilde{\beta},\chi}(x;h), e^{i \varphi_{\tilde{\beta}'}(x)/h}
a_{\tilde{\beta}',\chi}(x;h) \rangle_{L^2}.
\end{aligned}
\]

Now, by (\ref{hurry8}) and Proposition \ref{nonstat}, each term in the second sum is a $O(h^\infty)$, and since the number of terms is bounded by some power of $h$, the second term is a $O(h^\infty)$. As to the first term, equation (\ref{sheriff3}) implies that it is bounded independently of $h$.
Therefore, $\|\chi (E_h-\mathrm{R}_h)\|_{L^2}$ is bounded independently of $h$. Since $\|\mathrm{R}_h\|_{L^2} = O\Big{(}h^{\frac{|\mathcal{P}(1)|}{4\sqrt{b_0}}-\varepsilon (1+ \mathcal{P}(1))}\Big{)}$, this concludes the proof.
\end{proof}

Theorem \ref{blacksabbath4} also allows us to characterize the semiclassical measure of $E_h$. The proof of the following corollary is very similar to that of Theorem \ref{th : L2}.

\begin{corolaire} \label{blacksabbath3}
We make the same hypotheses as in Theorem \ref{blacksabbath4}.
Let $\chi\in C_c^\infty(X)$ and let $\varepsilon\bel 0$.
Then there exists a finite measure $\mu_\chi$ on $S^*X$ such that we have for any  $\psi\in S^{comp}(T^*X)$
\begin{equation*}\langle Op_h(\psi) \chi E_h, \chi E_h\rangle = \int_{T^*X} \psi(x,\xi)
\mathrm{d}\mu_{\chi}(x,\xi) +
O\Big{(}h^{\min\big{(}1, \frac{|\mathcal{P}(1)|}{4\sqrt{b_0}}-\epsilon \big{)}}\Big{)},
\end{equation*}

Furthermore, if $\chi$ has a small enough support, the measure $\mu_\chi$ is then given by
\begin{equation}\label{eq:ExprMeasure}
\mathrm{d}\mu_{\chi}(x,\xi) = \sum_{\tilde{\beta}\in
\tilde{\mathcal{B}}^\chi} | a_{\tilde{\beta},\chi}^0|^2(x) \delta_{\{\xi=\partial
\varphi_{\tilde{\beta}}(x)\}} \mathrm{d} x,
\end{equation}
where $a_{\tilde{\beta}}$ is as in (\ref{qad2}), and $a_{\tilde{\beta}}^0$ is its principal symbol as defined in Definition \ref{defsymbclassique}.
\end{corolaire}

\begin{proof}
Once again, it suffices to prove the result when $\chi$ has a small enough support. As in the proof of Theorem \ref{th : L2}, we can find $M^\varepsilon$ such that
\[
\begin{aligned}
&\langle Op_h(\psi) \chi E_h, \chi E_h\rangle
=\sum_{\substack{\tilde{\beta} \in \tilde{\mathcal{B}}^\chi\\
                \tilde{n}(\tilde{\beta})\leq M^{\varepsilon}|\log h|}}  \langle Op_h(\psi) e^{i \varphi_{\tilde{\beta}}(x)/h}
a_{\tilde{\beta},\chi}(x;h), e^{i \varphi_{\tilde{\beta}}(x)/h}
a_{\tilde{\beta},\chi}(x;h) \rangle \\
&+ \sum_{\substack{\tilde{\beta}\neq \tilde{\beta}'  \in \tilde{\mathcal{B}}^\chi\\
                \tilde{n}(\tilde{\beta}), \tilde{n}(\tilde{\beta}')\leq M^{\varepsilon}|\log h|}}  \langle Op_h(\psi) e^{i \varphi_{\tilde{\beta}}(x)/h}
a_{\tilde{\beta},\chi}(x;h), e^{i \varphi_{\tilde{\beta}'}(x)/h}
a_{\tilde{\beta}',\chi}(x;h) \rangle +  O\Big{(}h^{\frac{|\mathcal{P}(1)|}{4\sqrt{b_0}}-\varepsilon (1+ \mathcal{P}(1))}\Big{)}.
\end{aligned}
\]

To deal with the second sum, recall that $Op_h(\psi) e^{i \varphi_{\tilde{\beta}}(x)/h}
a_{\tilde{\beta},\chi}(x;h) = e^{i \varphi_{\tilde{\beta}}(x)/h}
b_{\tilde{\beta},\chi}(x;h)$ for some symbol $b_{\tilde{\beta},\chi} \in S^{comp}(X)$. See for instance \cite[Proposition 3.5]{DG} for a proof of this fact. Hence, the proof of Theorem \ref{th : L2} shows that the second sum is $O(h^\infty)$. 

 As to the first sum, we know from \cite[\S 5.1, Example 2]{Zworski_2012} that 
$$\langle Op_h(\psi) e^{i \varphi_{\tilde{\beta}}(x)/h}
a_{\tilde{\beta},\chi}(x;h), e^{i \varphi_{\tilde{\beta}}(x)/h}
a_{\tilde{\beta},\chi}(x;h) \rangle = \int_{T^*X} \psi(x,\xi) | a_{\tilde{\beta},\chi}^0|^2(x) \delta_{\{\xi=\partial
\varphi_{\tilde{\beta}}(x)\}} + O(h \|a_{\tilde{\beta},\chi} \psi\|_{C^{3+2d}}).$$

The result then follows from (\ref{sheriff3}), which also gives us the convergence of the right-hand side in (\ref{eq:ExprMeasure}).
\end{proof}

Finally, Theorem \ref{blacksabbath4} allows us to obtain some bounds on the $L^p$ norms of $E_h$.

\begin{corolaire}\label{cor:borne Lp}
We make the same assumptions as in Theorem \ref{blacksabbath4}. 

Let $r_h:= h^\alpha$, where
\begin{equation}\label{eq:DefAlpha1}
\alpha<\frac{1}{2\left( 1+ \frac{\sqrt{b_0} d}{|\mathcal{P}(1)|}\right)}.
\end{equation}

There exists $C_{\chi,\alpha}$ such that for all $h$ small enough, we have
\begin{equation}\label{eq: borne Linfini}
\|\chi E_h\|_{L^\infty} \leq C_{\chi,\alpha} \Big{(}\frac{h}{r_{h}}\Big{)}^{-(d-1)/2}.
\end{equation}
Let $$p_d := \frac{2(d+1)}{d-1}.$$ There exists $C'_{\chi,\alpha}>0$ such that
\begin{equation}\label{eq:borneLp}
\|\chi E_h\|_{L^{p_d}}\leq C'_{\chi,\alpha} \Big{(}\frac{h}{r_{h}}\Big{)}^{-1/p_d}.
\end{equation}
\end{corolaire}

\begin{remarque}
One could also obtain $L^p$ bounds on $\chi E_h$ for any $p>2$, by interpolation. The special value $p_d$ corresponds to the critical exponent in the Sogge inequalities (see \cite[\S 10.4]{Zworski_2012} and the references therein). Actually, our estimates are analogous to the Sogge estimates for eigenvalues on compact manifolds, except that in the right hand side, the negative power of $h$ is replaced by the same power of $h/r_{h}$. We hence have an improvement of some power of $h$.
\end{remarque}

Corollary \ref{cor:borne Lp} will follow from the following small-scale $L^2$ bound, which has interest on its own. If $x\in X$ and $r>0$, we shall write $B(x,r)$ for the geodesic ball of center $x$ and of radius $r$.

\begin{proposition}\label{prop: small scale}
Let $r_h:= h^\alpha$, where $\alpha$ is as in \eqref{eq:DefAlpha1}.
For any $x_0\in X$, there exists $C>0$ such that
\begin{equation}
\int_{B(x_0,r_h)} |E_h|^2 \leq C r_h^{d}.
\end{equation}
Here, the constant $C$, can be taken independent of $x_0$ if $x_0$ varies in a compact set.
\end{proposition}

\begin{remarque}
Using the methods of \cite[\S 6]{Ing2}, it is possible to show that $\int_{B(x_0,r_h)} |E_h|^2$ is also bounded from below by some constant times $r_h^{d}$. Since we won't need the lower bound for our $L^p$ estimates, we will not give a proof here.
\end{remarque}

We will prove Proposition \ref{prop: small scale} and Corollary \ref{cor:borne Lp} in section \ref{sec: proof Lp}.

\section{Facts from classical dynamics}\label{sec:tools}

In this section, we shall recall a few constructions from classical dynamics which will be useful in the proof of Theorem \ref{blacksabbath4}.

Let us start with the alternative definition of topological pressure given in \cite{NZ}.
\subsection{A useful definition of topological pressure}\label{sec: other press}

For any $0<\delta<1$, we define
\begin{equation*}
\begin{aligned}
\mathcal{E}^\delta &:= \{(x,\xi)\in T^*X; |\xi|^2 \in (1-\delta, 1+\delta)\},\\ 
K^\delta &:= \{(x,\xi)\in \mathcal{E}^\delta; \Phi^t(x,\xi) \text{ remains in a bounded set for all } t\in \mathbb{R}\}.
\end{aligned}
\end{equation*}

Let $\mathcal{W}=(W_a)_{a\in A_1}$ be a finite open cover of
$K^{\delta/2}$, such that the $W_a$ are all strictly included in
$\mathcal{E}^\delta$.
For any $T\in \mathbb{N}^*$, define $W(T):=(W_\alpha)_{\alpha \in A_1^T}$ by
\begin{equation*}W_\alpha := \bigcap_{k=0}^{T-1} \Phi^{-k}(W_{a_k}),
\end{equation*} where
$\alpha=(a_0,..,a_{T-1})$.
Let $\mathcal{A}'_T$ be the set of $\alpha\in A_1^T$ such that $W_\alpha\cap
K^\delta \neq \emptyset$.
If $V\subset \mathcal{E}^ \delta$, $V\cap K^{\delta/2} \neq \emptyset$,
define
\begin{equation*}S_T(V) := -\inf_{\rho\in V\cap K^{ \delta/2}} \lambda^+_T(\rho),~~~~~\text{  with } \lambda_T^+\text{ as in (\ref{defJaco})}.
\end{equation*}
\begin{equation*}Z_T(\mathcal{W},s):=\inf \big{\{ } \sum_{\alpha\in \mathcal{A}_T} \exp
\{s S_T (W_\alpha)\} : \mathcal{A}_T\subset \mathcal{A'}_T, K^{\delta/2}
\subset \bigcup_{\alpha \in \mathcal{A}_T} W_\alpha \big{\} }
\end{equation*}
\begin{equation*} \mathcal{P}^\delta (s):= \lim \limits_{diam \mathcal{W} \rightarrow
0} \lim \limits_{T \rightarrow \infty} \frac{1}{T} \log
Z_T(\mathcal{W},s).
\end{equation*}
The topological pressure is then:
\begin{equation} \label{Franck}
\mathcal{P} (s) = \lim \limits_{\delta \rightarrow 0} \mathcal{P}^
\delta (s).
\end{equation}

Let us fix $\epsilon_0>0$ so that $\mathcal{P}(1)+2\epsilon_0 < 0$. Then there exists
$t_0>0$, and $\hat{\mathcal{W}}$ an open cover of $K^\delta$
with $diam (\hat{\mathcal{W}})<\secur$ such that
\begin{equation}\label{Tea}
\Big{|} \frac{1}{t_0} \log Z_{t_0} (\hat{\mathcal{W}},s) -
\mathcal{P}^\delta (s) \Big{|} \leq \epsilon_0.
\end{equation}
We can find $\mathcal{A}_{t_0}$ so that $\{W_\alpha, \alpha \in
\mathcal{A}_{t_0} \}$ is an open cover of $K^\delta$ in
$\mathcal{E}^\delta $ and such that
\begin{equation*} \sum_{\alpha\in \mathcal{A}_{t_0}} \exp \{s S_{t_0} (W_\alpha) \} \leq
\exp \{ t_0 (\mathcal{P}^\delta (s) + \epsilon_0) \} .
\end{equation*}
Therefore, if we take $\delta$ small enough, and if we rename $\{W_\alpha,
\alpha \in \mathcal{A}_{t_0} \}$ as $\{V_b, b \in B_1\}$, we have:
\begin{equation}\label{rainy} \sum_{b\in B_1} \exp \{ S_{t_0} (V_b) \} \leq \exp \{ t_0
(\mathcal{P}(1) + 2\epsilon_0) \} .
\end{equation}

\subsection{Truncated propagation of Lagrangian manifolds}
Let us recall the results given in \cite[\S 2]{Ing} and \cite[\S 4]{Ing2} concerning the propagation of $\Lambda_\xi$ by the geodesic flow. To do this, we have to introduce a nice open covering of $S^*X$.

We may find a finite number of open sets $(V_b)_{b\in B_2 \sqcup \{0\}}$ with $V_{0} = T^*(X\backslash X_0)\cap \mathcal{E}^{\delta}$, with $V_b\cap K = \emptyset$ and $V_b$ is bounded if $b\in B_2$, such that $(V_b)_{b\in B}$ is an open cover of $S^*X$ included in $\mathcal{E}^\delta$,  with $B= B_1\sqcup B_2 \sqcup \{0\}$, and such that the following holds.

\subsubsection*{Truncated Lagrangians}
Let
$N\in\mathbb{N}$, and
let $\beta=(\beta_0,\beta_1,...,\beta_{N-1})\in B^{N}$.
Let $\Lambda$ be a Lagrangian manifold
in $T^*X$. We define the sequence of (possibly empty) Lagrangian manifolds
$(\Phi_\beta^{k}(\Lambda))_{0\leq k \leq N-1}$ by recurrence by:
\begin{equation}\label{deftronque}
\Phi_\beta^{0}(\Lambda)= \Lambda \cap V_{\beta_0},~~~~
\Phi_\beta^{k+1}(\Lambda) =
V_{\beta_{k+1}} \cap  \Phi^{t_0} (\Phi_{\beta}^{k}
(\Lambda)).
\end{equation}

If $\beta\in B^N$, we will often write 
\begin{equation*}
\Phi_\beta (\Lambda) := \Phi_\beta^{N-1}(\Lambda). 
\end{equation*}

For any $\beta\in B^{N}$ such that $\beta_{N-1}\neq 0$, we will define
\begin{equation}\label{rollingstone}
\tau(\beta):= \max \{1\leq i \leq N-1; \beta_i=0\}
\end{equation}
 if there exists 
$1\leq i \leq N-1$ with $\beta_i=0$, and $\tau(\beta)=0$ otherwise.

In the sequel, we will be interested in the truncated propagation of the manifolds $\Lambda_\xi$ indexed by $\xi\in \mathbb{S}^{d-1}$, defined as follows. If $(X,g)$ is hyperbolic near infinity, then $\Lambda_\xi$ is defined as in (\ref{eq:DefLag}), while if $(X,g)$ is Euclidean near infinity, then
\begin{equation}\label{eq:DefLagEucl}
\Lambda_\xi := \{(x,\xi); x\in X\backslash X_0\} \subset S^*X.
\end{equation}

By possibly taking the set $X_0$ bigger, we can assume that 
\begin{equation}\label{invariance}
\forall \xi \in \mathbb{S}^{d-1},~~ \Phi_{0,...,0}(\Lambda_\xi) = \Lambda_\xi.
\end{equation}

The main results of \cite{Ing2} concerning the truncated propagation of $\Lambda_\xi$ can be summed up as follows.

\begin{theoreme}\label{th: rappel classik}
Let $(X,g)$ be a a manifold of non-positive sectional curvature which is Euclidean or hyperbolic near infinity, whose trapped set is hyperbolic, and let $\xi \in \mathbb{S}^{d-1}$.

Let $\mathcal{O}\subset X$ be an open set which is small enough so that we may define local coordinates on it. Then the manifolds $\Phi_\beta(\Lambda_\xi)\cap T^*\mathcal{O}$ satisfy the following properties.

\begin{enumerate}
\item (Finite time away from the trapped set) Let $\mathcal{O'}$ be a bounded open set. There exists $N_{\mathcal{O},\mathcal{O'}}$ such that, for all $N\in \mathbb{N}$ and $\beta\in B^N$, if $\Phi_\beta (\Lambda_\xi \cap \mathcal{O}') \cap \mathcal{O}\neq \emptyset$, then we have
\begin{equation*}
\begin{aligned}
& i\in \{N_{\mathcal{O},\mathcal{O'}},..., N-N_{\mathcal{O},\mathcal{O'}}\} \Longrightarrow \beta_i \in B_1.
\end{aligned}
\end{equation*}

\item (Smooth projection) For any $N\in \mathbb{N}$ and any $\beta\in B^N$, $\Phi_\beta(\Lambda_\xi)\cap (S^*\mathcal{O})$ is a Lagrangian manifold which may be written, in local coordinates, in the form
\begin{equation*}
\Phi_\beta(\Lambda_\xi)\cap T^*\mathcal{O}\equiv \{(x,\partial_x \varphi_{\beta, \mathcal{O}} (x)); x\in \mathcal{O}^\beta\},
\end{equation*}
where $\mathcal{O}^\beta$ is an open subset of $\mathcal{O}$.

Furthermore, for any $\ell\in \mathbb{N}$, there exists a $C_{\ell,\mathcal{O}}\bel 0$ such that for any $N\in \mathbb{N}$, $\beta\in B^N$, we have
\begin{equation}\label{bogota}
\|\partial_x \varphi_{\beta,\mathcal{O}}\|_{C^\ell}\leq C_{\ell,\mathcal{O}}.
\end{equation}

\item (Expansion)
If $x\in \mathcal{O}^\beta$, let us write
\begin{equation}
\Phi^{-Nt_0}(x,\partial_x \varphi_{\beta, \mathcal{O}} (x)) = (g_{\beta,\mathcal{O}}(x), \xi).
\end{equation}

Set \begin{equation}
J_{N,\beta,\mathcal{O}} := \sup\limits_{x\in \mathcal{O}^\beta} \det |g_{\beta,\mathcal{O}}|^{1/2}.
\end{equation}

If $N\in \mathbb{N}$ and $\beta \in B^N$, we shall write $\beta^1$ for the sequence obtained by keeping only the elements of $\beta$ which belong to $B_1$. We have $\beta^1 \in (B_1)^{n_1(\beta)}$ for some $n_1(\beta)\leq N$.

There exists $C>1$ and $N_0>0$ such that for all $N\in \mathbb{N}$ and all $\beta \in B^N$, we have
\begin{equation}\label{eq:BorneJ}
J_{N,\beta,\mathcal{O}} \leq C \exp\Big{[}\sum_{i=0}^{n_1(\beta)} \frac{S_{t_0}(V_{\beta^1_i})}{2}\Big{]}.
\end{equation}

\item (Distance between the Lagrangian manifolds)
There exists a constant $C_\mathcal{O}\bel 0$ such that for any $n,n'\in \mathbb{N}$, for any $\beta\in B^n, \beta'\in B^{n'}$, and for any $x\in\mathcal{O}$, such that $x\in \pi_X\big{(}\Phi_{\beta,\mathcal{O}}(\Lambda_\xi)\big{)}\cap  \pi_X\big{(}\Phi_{\beta',\mathcal{O}}(\Lambda_\xi)\big{)}$,
we have either $\partial \varphi_{\beta,\mathcal{O}}(x) = \partial \varphi_{\beta',\mathcal{O}}(x) $ or 
\begin{equation}\label{hurry}
|\partial \varphi_{\beta,\mathcal{O}} (x)- \partial \varphi_{\beta',\mathcal{O}}(x)|\geq C_1 e^{- t_0 \sqrt{b_0} \max(n-\tau(\beta),n'-\tau(\beta'))}, 
\end{equation}
with $\tau(\beta)$ defined as in (\ref{rollingstone}), and where $b_0$ is as in \eqref{eq: def b0}.
\end{enumerate}
\end{theoreme}

The proof of 1. is given at the beginning of the proof of \cite[Th 3.3]{Ing2}, 2. is \cite[Th 3.3]{Ing2}; 4. is \cite[Lemma 4.10]{Ing2}, while 3. follows from \cite[Proposition 5.2]{NZ}. Note that, in \cite{Ing2}, the author made a rescaling of time to suppose that $t_0=1$.

\section{Proof of Theorem \ref{blacksabbath4}}\label{sec:proof}
In the rest of the paper, we will denote by $U(t) = e^{ith \Delta}$ the semiclassical Schrödinger propagator, and 
$$\tilde{U}(t):= e^{i\frac{t}{h} (h^2\Delta+1)}.$$

The starting point of the proof will be the following decomposition, which is proven in \cite[\S 5.1]{Ing}.
 For any $\chi_1 \in C_c^\infty(X; [0,1])$ such that $\chi_1\equiv 1$ on $X_0$ and any $t_0>0$ large enough, there exists $\chi_{t_0}\in C_c^\infty(X)$ such that, for any $M>0$, the following holds for any $N t_0\leq  M |\log h |$:

\begin{equation}\label{canardsauvage}
\chi_1 E_h =(\chi_1 \tilde{U}(t_0))^N \chi_{t_0} E_h+ \sum_{k=1}^{N} (\chi_1 \tilde{U}(t_0))^k
(1-\chi_1)\chi_{t_0} E^0_h +O_{L^2}(h^\infty).
\end{equation}

The proof of (\ref{canardsauvage}) in \cite{Ing}, which is based on \cite[Lemma 3.10]{DG} did not use the hypothesis on the topological pressure, but merely the hypothesis that $E_h$ is a tempered distribution which follows from the fact that the resolvent is bounded polynomially on the real axis.

The following proposition and its proof are inspired by \cite[\S 3.2]{dyatlov2017fractal}. We will prove it in section \ref{subsec:NoOutside} below. It will be used in section \ref{subsec:Suite} to replace the terms involving $(\chi_1 \tilde{U}(t_0))^k$ in (\ref{canardsauvage}) by terms involving $U( k t_0)$, which are easier to bound using our bounds on the resolvent.

\begin{proposition}\label{prop:NoOutside}
If $t_0$ is chosen large enough, then the following holds. Let $\psi\in C^\infty_c\left( \frac{2}{3}, \frac{3}{2} \right)$. For any $M>0$ and any $N\in \N$ with $Nt_0 \leq M |\log h|$, we have
\begin{equation}\label{eq:NoOutside}
\chi_1 U(N t_0) (1-\chi_1) U(t_0) \chi_1 \psi (-h^2 \Delta) = O_{L^2\to L^2}(h^\infty).
\end{equation}
\end{proposition}

\subsection{Proof of Proposition \ref{prop:NoOutside}}\label{subsec:NoOutside}
\subsubsection{Preliminary constructions}
Since $X$ is hyperbolic or Euclidean near infinity, we may define a boundary defining function $b$ outside of the compact set $X_0$: in the hyperbolic near infinity case, it comes from Definition \ref{def:HypInf}, while in the Euclidean near infinity case, we take $\frac{1}{|x|}$ in the Euclidean part. If $\rho \in T^*X$, we say that $\rho$ is \emph{outgoing} if, for all $t\geq 0$, we have $\pi_X(\Phi^t \rho)\notin X_0$, and the map $[0, + \infty) \ni t \mapsto b \left( \pi_X \Phi^t \rho \right)$ is decreasing.

If $\chi\in C_c^\infty(X)$ satisfies $\chi \equiv 1$ on $X_0$, then there exists $t_0 = t_0(\chi)>0$ such that, for any $\rho = (x,\xi)\in T^*X$ with $|\xi|\in \left[ \frac{1}{2}, 2 \right]$ and $x\in \mathrm{supp}(\chi)$, we have
\begin{equation}\label{eq:TransportSupport}
\pi_X\left(\Phi^{t_0}(\rho)\right) \in \mathrm{supp}(1-\chi) \Longrightarrow \Phi^{t_0}(\rho) \text{ is outgoing and $\forall t \geq 2t_0$, $\pi_X\left(\Phi^{t}(\rho)\right) \notin \mathrm{supp} (\chi)$.}
\end{equation}

This follows from the fact that the projection of the trapped set on the base manifold cannot intersect $\mathrm{supp}(1-\chi)$. 

Let $\chi_1 \in C_c^\infty(X)$ be such that (\ref{canardsauvage}) holds, and let $\chi_2 \in C_c^\infty(X)$ be such that $\mathrm{supp} (\chi_2)$ contains a neighborhood of $\mathrm{supp}(\chi_1)$ and $\mathrm{supp} (1-\chi_2)$ contains a neighborhood of $\mathrm{supp}(1-\chi_1)$.

Let $t_0= t_0(\chi_2)$ be such that (\ref{eq:TransportSupport}) holds for $\chi=\chi_2$. Then (\ref{eq:TransportSupport}) also holds for this $t_0$ and for $\chi = \chi_1$.

Let us write, for $j=1,2$
$$P_{\chi_j} := \left\{\rho= (x,\xi) \in T^* X ; |\xi|\in \big[ \frac{2}{3} - \frac{j}{20}, \frac{3}{2} + \frac{j}{20} \big], x \in \mathrm{supp}(1-\chi_j) \text{ and } \pi_X \left( \Phi^{-t_0}(\rho)\right) \in \mathrm{supp } (\chi_j) \right\}.$$
This is a closed set, $P_{\chi_2}$ contains a neighborhood of $P_{\chi_1}$, and thanks to \eqref{eq:TransportSupport}, we have
\begin{equation}\label{eq:PSort}
\rho\in P_{\chi_j} \Longrightarrow \rho \text{ is outgoing and } \forall t\geq t_0,~  \pi_X\left(\Phi^{t}(\rho)\right) \notin \mathrm{supp} (\chi).
\end{equation}

The proof of Proposition \ref{prop:NoOutside} will follow from the following lemma.

\begin{lemme}\label{lem:Away}
Let $a\in C_c^\infty(T^*X)$ be supported in $P_{\chi_2}$. Then, for any $M>0$ and any $N\in \N$ with $N t_0 \leq M |\log h|$, we have
$$\chi_1 U(N t_0) Op_h(a) = O_{L^2\to L^2}(h^\infty).$$
\end{lemme}

Indeed, let $a \in C_c^\infty(T^*X)$ be supported in $P_{\chi_2}$ with $a\equiv 1$ on $P_{\chi_1}$, and let $\psi\in C^\infty_c\left( \frac{2}{3}, \frac{3}{2} \right)$. It follows from the definition of $P_{\chi_1}$ and from Egorov's theorem (Proposition \ref{prop:Egorov}) that

$$(1-\chi_1) U(t_0) \chi_1 \psi (-h^2 \Delta) = Op_h(a) (1-\chi_1) U(t_0) \chi_1 \psi (-h^2 \Delta) + O_{L^2\to L^2}(h^\infty),$$
so that Proposition \ref{prop:NoOutside} follows. The rest of section \ref{subsec:NoOutside} will be devoted to the proof of Lemma \ref{lem:Away}.

\subsubsection{A parametrix for the propagator}

If $\mathcal{V}\subset T^*X$ is an open set, we shall write 
$$\mathcal{Z}(\mathcal{V}) := \{ \pi_X \Phi^t (\rho') ; t\geq 0, \rho'\in \mathcal{V}\}\subset X.$$

Let $\rho \in P_{\chi_2}$.
 Since $\rho$ is outgoing, we may find an open neighborhood $\mathcal{V}_0$ of $\rho$ such that $\mathcal{Z}(\mathcal{V}_0)$ is simply connected.
 
We may thus define a map
$\Psi : \mathcal{Z}(\mathcal{V}_0) \longrightarrow Y_0$ which is isometric onto its image $\widehat{\mathcal{Z}_0}$, where $Y_0 = \R^d$ if $X$ is Euclidean near infinity, and $Y_0 = \mathbb{H}^d$ if $X$ is hyperbolic near infinity.

Let $\mathcal{V}_1$ and $\mathcal{V}_2$ be open sets containing $\rho$, such that $\overline{\mathcal{V}_2} \subset \mathcal{V}_1\subset \overline{\mathcal{V}_1}\subset \mathcal{V}_0$. 

We shall write \begin{equation}\label{eq:BorneInfDist}
c_1 := d \left( \mathcal{Z}(\mathcal{V}_2), X \backslash \mathcal{Z}(\mathcal{V}_1) \right) = d \left( \mathcal{Z}(\mathcal{V}_2),  \mathcal{Z}(\mathcal{V}_0) \backslash \mathcal{Z}(\mathcal{V}_1) \right)>0.
\end{equation}

To see that $c_1>0$, we start by noting that
there exists $\varepsilon>0$ such that if $\rho\in \mathcal{V}_2$ and $\rho'\in T^*X$ are such that $d(\pi_X \Phi^t(\rho), \Phi^t(\rho'))<\varepsilon$ for all $t\in [-1,0]$, then we have $\rho'\in \mathcal{V}_1$.

Let $x \in \mathcal{Z}(\mathcal{V}_2)$, so that there exists $\xi \in T_x^*X$ and $\tau \leq 0$ with $\Phi^{\tau}(x,\xi) \in \mathcal{V}_2$. Let $x'\in \mathcal{Z}(\mathcal{V}_0) \backslash \mathcal{Z}(\mathcal{V}_1)$ be such that the distance between $x$ and $x'$ is smaller than the injectivity radius of $X$. Since the universal cover of $X$ has non-positive curvature, we deduce that there exists a (unique) $\xi'\in T_{x'}^*X$ such that the map $(-\infty, 0]\ni t\mapsto d\left( \pi_X \Phi^t (x,\xi), \pi_X \Phi^t (x',\xi') \right)$ is non-decreasing.

In particular, if $d(x,x')<\varepsilon$, then we have $d\left( \pi_X \Phi^{\tau +t} (x,\xi), \pi_X \Phi^{\tau+t} (x',\xi') \right)<\varepsilon$ for all $t\in [-1,0]$, so that $\Phi^\tau(x',\xi')\in \mathcal{V}_1$, which contradicts the fact that $x'\notin \mathcal{Z}(\mathcal{V}_1)$. Hence, we must have $d(x,x')>\varepsilon$, which gives us (\ref{eq:BorneInfDist}).

Since we also have $d \left( \mathcal{Z}(\mathcal{V}_1), X \backslash \mathcal{Z}(\mathcal{V}_0) \right)>0$, we may find $\psi_0\in C^\infty(\mathcal{Z}(\mathcal{V}_0))$ equal to one on $\mathcal{Z}(\mathcal{V}_1)$, and such that $\|\psi_0\|_{C^2(\mathcal{V}_0)} < + \infty$.

Note that $ P_{\chi_2}$ is compact, so that it may be covered by finitely many sets $\mathcal{V}_2$ as above. By using a partition of unity argument, and the fact that the quantization $Op_h$ is linear, we may suppose that the function $a$ appearing in Lemma \ref{lem:Away} satisfies $\mathrm{supp}(a) \subset \mathcal{V}_2$.

Let $U^0(t)$ denote the free Schrödinger propagator on $Y_0$. The operator $\psi_0 (\Psi^{-1})^* U^0(t) \Psi^* \psi_0 : L^2(X) \longrightarrow L^2(X)$ is thus well-defined.

 We claim that, for any $M>0$ and any $0\leq t \leq M |\log h|$, we have
\begin{equation}\label{eq:Parametrix}
 U(t) Op_h(a) = \psi_0 (\Psi^{-1})^* U^0(t) \Psi^* \psi_0  Op_h(a) + O_{L^2 \to L^2} (h^\infty),
\end{equation}
with the remainder bounded independently of $t$.

We have $i \frac{\mathrm{d}}{\mathrm{d}t}  \left[\psi_0 (\Psi^{-1})^* U^0(t) \Psi^* \psi_0 \right] = -h^2 \Delta \psi_0 (\Psi^{-1})^* U^0(t) \Psi^* \psi_0 + \left[ \psi_0, h^2 \Delta \right] (\Psi^{-1})^* U^0(t) \Psi^* \psi_0$. Hence, by Duhamel's principle, we have
$$\psi_0 (\Psi^{-1})^* U^0(t) \Psi^* \psi_0 = U(t) + \int_0^t U(t-s) \left[ \psi_0, h^2 \Delta \right] (\Psi^{-1})^* U^0(s) \Psi^* \psi_0 \mathrm{d}s.$$

Therefore, (\ref{eq:Parametrix}) will follow if we can show that, for any $a\in C_c^\infty (\mathcal{V}_2)$, any $k\in \N$ and any $t\in [ h^k, M |\log h| ]$, we have
\begin{equation}\label{eq:BabyParametrix}
\left[ \psi_0, h^2 \Delta \right] (\Psi^{-1})^* U^0(s) \Psi^* \psi_0 Op_h(a) = O_{L^2 \to L^2}(h^\infty).
\end{equation}

We will also show that, for any $N\in \N$, with $Nt_0 \leq M |\log h|$, we have
\begin{equation}\label{eq:BabyParametrix2}
\chi_1 \psi_0 (\Psi^{-1})^* U^0(N t_0) \Psi^* \psi_0 Op_h(a) = O_{L^2 \to L^2}(h^\infty).
\end{equation}

Lemma \ref{lem:Away} then follows from equations (\ref{eq:Parametrix}) and (\ref{eq:BabyParametrix2}), and hence, from equations \eqref{eq:BabyParametrix} and \eqref{eq:BabyParametrix2}. Proving these two equations is the goal of the next paragraph.

\subsubsection{Long time estimates for the free propagator}

Let us denote by $u_h^0(t,x,y)$ the integral kernel of $U^0(t)$. It takes the form $u_h^0(t, d_{Y_0}(x,y))$, where
\begin{equation*}
\begin{aligned}
u^0_h(t,\delta) &= \frac{1}{(4 i \pi t h)^{d/2}} e^{\frac{i}{h} \frac{\delta^2}{4t}} &\text{ when $Y_0= \mathbb{R}^d$}\\
u^0_h(t, \delta) &= \frac{c}{\sqrt{|th|}} e^{-ith \frac{(d-1)^2}{4}} \left( \frac{\partial_\delta}{\sinh \delta} \right)^{\frac{d-1}{2}} e^{\frac{i}{h} \frac{\delta^2}{4t}}  ~~ &\text{ when $Y_0 = \mathbb{H}^d$ with $d$ odd }\\
 u^0_h(t, \delta)&= \frac{c}{|th|^{3/2}} e^{-ith \frac{(d-1)^2}{4}} \left(  \frac{\partial_\delta}{\sinh \delta} \right)^{\frac{d-2}{2}} \int_{0}^{+\infty} \frac{e^{\frac{i}{h} \frac{(s+\delta)^2}{4t}} (s+\delta)}{\sqrt{\cosh (s+\delta) - \cosh \delta}} \mathrm{d}s  ~~ &\text{ when $Y_0 = \mathbb{H}^d$ with $d$ even.}
 \end{aligned}
 \end{equation*}
The expression of the free Schrödinger propagator on $\R^d$ is standard, while the expression on $\mathbb{H}^d$ was first proven in \cite{Ban07}. 

Therefore, thanks to (\ref{eq:DefQuantification}), we see that the operator $(\Psi^{-1})^* U^0(s) \Psi^* \psi_0 Op_h(a)$ has integral kernel
\begin{equation}\label{eq:IntKernel}
K_h(x,z; t)= \frac{1}{(2\pi h)^d} \sum_j \int_{T^* X} u_h(t,x,y)  e^{\frac{i}{h} (\gamma_j(y)-\gamma_j(z))\cdot (\partial \gamma_j (y)^T) (\xi) } \chi_j(y)\chi_j(z) a(y, \xi )g_{\gamma_j}(z,\xi) \mathrm{d}y \mathrm{d}\xi,
\end{equation}
where the $\chi_j$ and the maps $\gamma_j$  depend on the choice made to define the quantization $Op_h$, as explained in Appendix \ref{greve}. The sum above contains only finitely many non-zero terms, and, to lighten notations, we will forget the dependence in $j$ in the sequel.

Let us write $K_h^{z,t}(x) := K_h(x, z ; t)$. Equation (\ref{eq:BabyParametrix}) will follow if we can prove that, for any $k\in \N$, any $M>0$, we have for all $t\in [ h^k, M |\log h| ]$
\begin{equation}\label{eq:NegligibleKernel}
\|K_h^{z,t}\|_{C^2 (\mathrm{supp}(\partial \psi_0))} = O(h^\infty),
\end{equation}
while, to prove (\ref{eq:BabyParametrix2}), it suffices to show that, for all $N\in \N$ with $Nt_0\leq M |\log h|$, we have
\begin{equation}\label{eq:NegligibleKernel2}
\|K_h^{z,Nt_0}\|_{C^0 (\mathrm{supp}(\chi_1 \psi_0))} = O(h^\infty),
\end{equation}

\subsubsection*{The case of $\R^d$ or $\mathbb{H}^d$ with $d$ odd}
For any $t>0$, we can write 
\begin{equation*}
K_h(x, z ; t) = c(t)  \int_{T^* X} b(t,x, z, \xi) e^{\frac{i}{h} \phi_{t,x,z}(y,\xi)} \mathrm{d}y \mathrm{d}\xi
\end{equation*}
with $|c(t)|=1$, the function $b$ being a finite sum of negative powers of $ht$ times smooth, compactly supported functions in $z$ and $\xi$, with all their derivatives bounded independently of $x\in X$, and the function $\phi_{t,x,z}(y,\xi)$ taking the form
$$\phi_{t,x,z}(y,\xi) =  (\gamma(y)-\gamma(z))\cdot (\partial \gamma (y)^T) (\xi) + \frac{d(x,y)^2}{4t}.$$

The gradient of $\phi_{t,x,z}$ can thus be written as 
$$\nabla \phi_{t,x,z}(y,\xi) = \left( \frac{d(x,y)}{2t} v_{x,y} + \hat{\xi}, \partial \gamma(y) (\gamma(y) - \gamma(z))\right),$$ 
where $\hat{\xi}$ is the vector of $T_yX$ corresponding to $\xi\in T^*_y X$ by the canonical identification, and where $v_{x,y}$ is the unit vector in $T_yX$ such that $-v_{x,y}$ points in the direction of $x$.

In particular, we see that, for $\nabla \phi_{t,x,z}$ to be zero, we must have $\pi_X \Phi^t(y,\xi)= x$. This cannot happen in the two situations we are interested in:
\begin{itemize}
\item $t>0$, $(y,\xi) \in \mathrm{supp} (a)$ and $x\in \mathrm{supp} (\partial \psi_0)$. Actually, it follows from (\ref{eq:BorneInfDist}) that we may find $c_2>0$ such that, for any $t>0$, $(y,\xi) \in \mathrm{supp} (a)$ and $x\in \mathrm{supp} (\partial \psi_0)$, we have
\begin{equation}\label{eq:LowerGrad}
\left|\nabla \phi_{t,x,z}(y,\xi) \right| \geq c_2 \left( 1 + \frac{d(x, \pi_X (\mathcal{V}_2))}{t} \right).
\end{equation}
\item $t\geq t_0$, $(y,\xi) \in \mathrm{supp} (a)$ and $x\in \mathrm{supp}(\chi_1)$. Indeed, it follows from (\ref{eq:TransportSupport}) and the definition of $P_{\chi_1}$ that there exists $c_3>0$ such that for all $(y,\xi) \in \mathrm{supp} (a)$, all $x\in \mathrm{supp}(\chi_1)$ and all $t\geq t_0$, we have
\begin{equation}\label{eq:LowerGrad2}
\left|\nabla \phi_{t,x,z}(y,\xi) \right| \geq c_3.
\end{equation}
\end{itemize}

Therefore, we deduce by integration by parts that for all $t> h^k$ for some $k>0$, and for all $x\in \mathrm{supp} (\partial \psi_0)$, we have

$$K_h(x, z ; t) = O(h^\infty),$$
with the remainder depending on $k$, but not on $x\in \mathrm{supp} (\partial \psi_0)$ or on $t$.

The derivatives of $K_h$ with respect to $x$ are of the form 
$$ \partial^\alpha_x K_h(x, z ; t) = c(t)  \int_{T^* X} b_\alpha(t,x, z, \xi) e^{\frac{i}{h} \phi_{t,x,z}(y,\xi)} \mathrm{d}y \mathrm{d}\xi,$$
where $b_\alpha$ is a finite sum of negative powers of $ht$ times smooth, compactly supported in $z$ and $\xi$, which are bounded, as well as all their derivatives, by powers of $(1 + \frac{d(x, \pi_X (\mathcal{V}_2))}{t})$. Hence, integrations by parts using (\ref{eq:LowerGrad}) give us (\ref{eq:NegligibleKernel}).

The proof of (\ref{eq:NegligibleKernel2}) is exactly the same, using (\ref{eq:LowerGrad2}) instead of (\ref{eq:LowerGrad}).

\subsubsection*{The case of $\mathbb{H}^d$ with $d$ even}
The proof is very similar as in the previous cases. For any $t>0$, we can write 
\begin{equation*}
K_h(x, z ; t) = c(t) \int_{0}^{+\infty} \mathrm{d}s  \int_{T^* X} b_s(t,x, z, \xi) e^{\frac{i}{h} \phi_{s,t,x,z}(y,\xi)} \mathrm{d}y \mathrm{d}\xi,
\end{equation*}
with $\phi_{s,t,x,z}(y,\xi) =  (\gamma(y)-\gamma(z))\cdot (\partial \gamma (y)^T) (\xi) + \frac{(s+d(x,y))^2}{4t}$, and $b_s$ being a finite sum of negative powers of $ht$ times compactly supported functions in $(z,\xi)$ which are bounded, as well as all their derivatives, polynomially in $s$ and $d(x, \pi_X \mathcal{V}_2)$.

We thus have the existence of $c_2>0$ such that for any $s,t>0$, $(y,\xi) \in \mathrm{supp} (a)$ and $x\in \mathrm{supp} (\partial \psi_0)$, we have
 $\left|\nabla \phi_{t,x,z}(y,\xi) \right| \geq c_2 \left( 1 + \frac{d(x, \pi_X (\mathcal{V}_2))}{t} + s \right)$. We may thus conclude that for all $t> h^k$ for some $k>0$, and for all $x\in \mathrm{supp} (\partial \psi_0)$, we have
$$\int_{T^* X} b_s(x, z, \xi) e^{\frac{i}{h} \phi_{s,t,x,z}(y,\xi)} \mathrm{d}y \mathrm{d}\xi \leq \frac{O(h^\infty)}{s^2},$$
with the $O(h^\infty)$ independent of $s$, of $x\in \mathrm{supp} (\partial \psi_0)$ and of $t$. Equation (\ref{eq:NegligibleKernel}) follows. The proof of (\ref{eq:NegligibleKernel2}) is very similar.

\subsection{Decay of the truncated propagator}\label{subsec:Suite}
Let $\chi\in C_c^\infty(X)$, and let $\chi_1 \in C_c^{\infty}$ be such that $\chi_1 \equiv 1$ on $X_0 \cup \mathrm{supp}(\chi)$. Let $t_0>0$ be large enough so that Proposition \ref{prop:NoOutside} holds.
 By (\ref{canardsauvage}), we may find $\chi_{t_0} \in C_c^\infty(X)$ such that 
\begin{equation}\label{canardsauvage12}
\chi E_h = \chi (\chi_1 \tilde{U}(t_0))^N \chi_{t_0} E_h+ \sum_{k=1}^{N} \chi (\chi_1 \tilde{U}(t_0))^k
(1-\chi_1)\chi_{t_0} E^0_h +O_{L^2}(h^\infty).
\end{equation}

The aim of this section is to show that, in equation (\ref{canardsauvage12}), the term $(\chi_1 \tilde{U}(t_0))^N \chi_{t_0} E_h$ is negligible. To this end, we start by stating a lemma which comes from \cite{ingremeau2017sharp}, and whose proof we recall. It shows that, if we wait for times which are large enough multiples of $|\log h|$, the propagator restricted to a compact set becomes smaller than any power of $h$. 

\begin{lemme}\label{lem:PropagDecay}
Let $\chi_1, \chi_2\in C_c^\infty(X)$, $\psi\in C_c^\infty(\R)$. For any $r>0$, there exists $M_{r}>0$ and $C_{r}>0$ such that 
$$\|\chi_1 U(M_{r} |\log h|)  \psi(-h^2 \Delta) \chi_2\|_{L^2\rightarrow L^2} \leq C_{r} h^{r}.$$  
\end{lemme}

\begin{proof}
Let us consider the incoming resolvent $R_-(z;h) := (-h^2\Delta- z)^{-1}$, which is analytic for $-\Im z >0$.
Using Stone's formula, we obtain that for any $t>0$, we have
\begin{equation*}
\chi_1 U(t)  \psi(-h^2 \Delta) \chi_2 = \frac{1}{2i \pi} \int_{\R} e^{-it z/h} \chi_1 \big{(}R_-(z;h) -  R_+(z;h)\big{)} \psi(z) \chi_2 \mathrm{d}z.
\end{equation*}

Let $\tilde{\psi}$ be an almost analytic extension of $\psi$, that is to say, a function $\tilde{\psi}\in C_c^\infty(\C)$ such that 
\begin{equation}\label{eq: almost analytic}
\partial_{\overline{z}} \tilde{\psi} (z) = O\big{(}(\Im z)^\infty\big{)}
\end{equation}
 and such that $\tilde{\psi}(z)= \psi(z)$ for $z\in \R$. We may furthermore assume that $$\mathrm{spt }~ \tilde{\psi}\subset \{z; \Re z \in \mathrm{spt}~\psi\}.$$
We refer the reader to \cite[\S 2]{martinez2002introduction} for the construction of such a function.

Using Green's formula, we obtain that 
\begin{equation*}
\begin{aligned}
\chi_1 U(t)  \psi(-h^2 \Delta) \chi_2 &= \frac{1}{2i \pi} \int_{\Im z = -C_0 h} e^{-it z/h} \chi_1 \big{(}R_-(z;h) -  R_+(z;h)\big{)} \tilde{\psi}(z) \chi_2 \mathrm{d}z\\
&+ \frac{1}{2i \pi} \int_{-C_0 h\leq \Im z \leq 0} e^{-it z/h} \chi_1 \big{(}R_-(z;h) -  R_+(z;h)\big{)} \partial_{\overline{z}} \psi(z) \chi_2 \mathrm{d}z.
\end{aligned}
\end{equation*}

Thanks to (\ref{eq:aprioribound}) and to (\ref{eq: almost analytic}), the second term is $O(h^\infty)$, independently of $t$. On the other hand, by (\ref{eq:aprioribound}), the first term is bounded by $C e^{-C_0 t } h^{-\alpha}$. Therefore, taking $t= M |\log h|$ with $M$ large enough proves the lemma.
\end{proof}

We would like to use Lemma \ref{lem:PropagDecay} to obtain bounds on $(\chi_1 \tilde{U}(t_0))^N \chi_{t_0} E_h$.

First of all, note that, since $\chi_1 \equiv 1$ on $\mathrm{supp}(\chi)$, we have 
\begin{equation}\label{eq:DecompoProp}
\chi U(N t_0) - \chi (\chi_1 U(t_0))^N = \sum_{\ell=0}^{N-2} \chi  U((N-\ell-1)t_0) (1-\chi_1)  U(t_0)  (\chi_1 U(t_0) )^\ell.
\end{equation}

Let $\varepsilon_0$ be as in section \ref{sec: other press}, small enough so that $(1-3 \varepsilon_0, 1+3 \varepsilon_0) \subset \left( \frac{2}{3}, \frac{3}{2}\right)$.
Let us fix a function $\psi\in C^\infty_c(1-\varepsilon_0, 1+\varepsilon_0)$ such that $\psi(x) = 1$ for $x\in (1-\varepsilon_0/2, 1+\varepsilon_0/2)$, and $\psi_1\in C^\infty_c(1- 3\varepsilon_0, 1+3\varepsilon_0)$ such that $\psi_1(x) = 1$ for $x\in (1-2\varepsilon_0, 1+2\varepsilon_0)$.

For any $1 \leq \ell \leq N-2$, we have (using Proposition \ref{prop:Egorov})
\begin{align*}
&\chi  U((N-\ell-1)t_0) (1-\chi_1)  U(t_0)  (\chi_1 U(t_0) )^\ell \psi(-h^2\Delta) \chi_{t_0}\\
 &= \chi  U((N-\ell-1)t_0) (1-\chi_1)  U(t_0) \chi_1 \psi_1(-h^2\Delta) U(t_0) (\chi_1 U(t_0) )^{\ell-1} \psi(-h^2\Delta) \chi_{t_0}  + O_{L^2\to L^2}(h^\infty)\\
&= O_{L^2\to L^2}(h^\infty)
\end{align*}
thanks to (\ref{eq:NoOutside}). We deduce from (\ref{eq:DecompoProp}) that
\begin{equation}\label{eq:WithOrWithoutCutoff}
\begin{aligned}
e^{iNt_0/h}\chi (\chi_1 \tilde{U}(t_0))^N \psi(-h^2\Delta) \chi_{t_0}  &= \chi U(N t_0) \psi(-h^2\Delta) \chi_{t_0} \\
&-  \chi_1  U((N-1)t_0) (1-\chi_1)  U(t_0)  \psi(-h^2\Delta)\chi_{t_0} + O(h^\infty).
\end{aligned}
\end{equation}

Therefore, the following corollary follows from (\ref{eq:WithOrWithoutCutoff}) and Lemma \ref{lem:PropagDecay}.
\begin{corolaire}\label{infinitif} Let $r>0$. We may find a constant
$M_{r}\geq 0$ such that for any $M > M_{r}$, for any $M_{r}|\log h|
\leq Nt_0 \leq M |\log h|$, we have:
\begin{equation*}\|\big{(}\chi_1 \tilde{U}(t_0)\big{)}^N
\psi(-h^2 \Delta) \chi_{t_0}\|_{L^2}=O(h^r).
\end{equation*}
\end{corolaire}

Now, recall that, by (\ref{eq: def Eout}) combined with (\ref{eq:aprioribound}), we have that $\|\chi_{t_0} E_h\|_{L^2} = O (h^{-\alpha})$ for some $\alpha>0$. Hence, the proof of \cite[Theorem 6.4]{Zworski_2012} shows that for any $\chi_2 \in C_c^\infty(X)$, we have
$$\big{(}1-\psi(-h^2\Delta) \big{)} \chi_2 E_h = O(h^\infty).$$

Therefore, we deduce from equation (\ref{canardsauvage12}) and Corollary \ref{infinitif} that
\begin{equation}\label{canardsauvage2}
\chi E_h = \sum_{k=1}^{M_r |\log h|} \chi (\chi_1 \tilde{U}(t_0))^k
(1-\chi_1)\chi_{t_0} E^0_h +O_{L^2}(h^r).
\end{equation}

We now have to decompose the propagators in order to take advantage of Theorem \ref{th: rappel classik}.

\subsection{Microlocal partition}\label{partition}
We take a partition of unity $\sum_{b\in B} \pi_b$ such that:
\begin{equation*}\sum_{b\in B} \pi_b(\rho)\equiv 1 \text { for all } \rho\in
\mathcal{E}^{\delta'},
\end{equation*} and $\mathrm{supp} (\pi_b) \subset V_b \subset
\mathcal{E}^\delta$ for all $b\in B$ and for some $0<\delta'<\delta$.

For $b\in B_1\cup B_2$, we set $\Pi_b:= Op_h(\pi_b)$. We have
\begin{equation*}WF_h(\Pi_b) \subset V_b\cap \mathcal{E}^\delta.
\end{equation*}

We then set 
\begin{equation*}
\Pi_0:= Id- \sum_{b\in B_1\cup B_2} \Pi_b.
\end{equation*}

We can decompose the propagator at time $t_0$ into:
\begin{equation*}\chi_1\tilde{U}(t_0)= \sum_{b\in B} \tilde{U}_b, \text{      where
 } \tilde{U}_b:=  \chi_1 \Pi_b e^{it_0/h} U(t_0).
\end{equation*}
The propagator at time $t_0$ may then be decomposed as follows:
\begin{equation}\label{triceratops}(\chi_1 \tilde{U}(t_0))^N= \sum_{\beta\in B^{N}}
\tilde{U}_{\beta},
\end{equation}
where $\tilde{U}_\beta := \tilde{U}_{\beta_{N-1}}\circ...\circ
\tilde{U}_{\beta_0}$. 

\subsection{Truncated propagation of Lagrangian states}
Let us fix an open set $\mathcal{O}\subset X$ small enough so that we may define local coordinates on it. Let us fix $\chi\in C_c^\infty(\mathcal{O})$, and choose $\chi_1$ equal to one on $\mathcal{O}$. From (\ref{canardsauvage2}) and (\ref{triceratops}), we have
\begin{equation}\label{canardsauvage4}
\chi E_h = \sum_{k=1}^{M_r |\log h|} \sum_{\beta \in B^k} \chi \tilde{U}_\beta
(1-\chi_1)\chi_{t_0} E^0_h +O_{L^2}(h^r).
\end{equation}

We may apply Proposition \ref{haroun} to describe $\chi\tilde{U}_\beta
(1-\chi_1)\chi_{t_0} E^0_h$. Indeed, the operator $\tilde{U}_\beta$ is of the form $(S_{i_N}\circ...\circ S_{i_1})$ with $S_{i_k}$ of the form (\ref{blabla}), and $E_h^0$ is a Lagrangian state associated to the Lagrangian $\Lambda_\xi$ defined by (\ref{eq:DefLag}) when $X$ is hyperbolic near infinity, and by (\ref{eq:DefLagEucl}) when $X$ is Euclidean near infinity.

 By the second point of Theorem \ref{th: rappel classik}, conditions (\ref{leffe}) and (\ref{eq:Caseprojettebien}) are satisfied. We obtain that

\begin{equation}
\chi \tilde{U}_\beta (1-\chi_1)\chi_{t_0} E^0_h = a_{\beta,\chi} e^{\frac{i}{h} \varphi_{\beta,\mathcal{O}}},
\end{equation}
with $a_{\beta,\chi}\in S^{comp}(\mathcal{O})$ a classical symbol which satisfies for any $\ell\in \mathbb{N}$
\begin{equation}\label{eq: borne C^0 itere}
\|a_{\beta,\chi}\|_{C^\ell}\leq C_\ell N^\ell J_{N,\beta,\mathcal{O}} \leq C_\ell N^\ell \exp\Big{[}\sum_{i=0}^{n_1(\beta)} \frac{S_{t_0}(V_{\beta^1_i})}{2}\Big{]},
\end{equation}
where $C_\ell$ does not depend on $\beta$ or $N$. Here, we used (\ref{eq:BorneJ}) in the second inequality.

In particular, we have
\begin{equation}\label{eq: L2 bound beta}
\|a_{\beta,\chi}\|^2_{C^\ell}\leq C'_\ell N^{2\ell} \exp\Big{[}\sum_{i=0}^{n_1(\beta)} S_{t_0}(V_{\beta^1_i})\Big{]}.
\end{equation}

Thanks to the first point in Theorem \ref{th: rappel classik} and to Proposition \ref{prop:Egorov}, we have that if $\beta\in B^k$ is such that $\chi \tilde{U}_\beta
(1-\chi_1)\chi_{t_0} E^0_h \neq O(h^\infty)$, then we must have $n_1(\beta)\geq  k-N_{\mathcal{O}}$ for some $N_\mathcal{O}>0$.

Therefore, for any $k\geq N_{\mathcal{O}}+1$, and any $\beta\in B^k$ such that $\chi \tilde{U}_\beta
(1-\chi_1)\chi_{t_0} E^0_h \neq O(h^\infty)$, $\beta$ must contain a word $\beta^1\in B_1^{k-N_{\mathcal{O}}}$. Furthermore, the number of $\beta$ corresponding to a given word $\beta^1$ is bounded independently of $k$. We deduce that
\begin{equation}\label{eq:BoundedWithPress}
\begin{aligned}
\sum_{\beta\in B^k} \|a_{\beta,\chi}\|^2_{L^2}&\leq \sum_{\beta^1 \in B_1^{k-N_\mathcal{O}}} ~~ \sum_{\beta\in B^k \text{containing $\beta^1$}} \|a_{\beta,\chi}\|^2_{L^2}\\
&\leq C \sum_{\beta^1 \in B_1^{k-N_\mathcal{O}}}~~ \sum_{\beta\in B^k \text{containing $\beta^1$}}  \exp\Big{[}\sum_{j=1}^{k-N_\mathcal{O}} S_{t_0}(V_{\beta^1_j})\Big{]}\\
&\leq  C' \sum_{\beta^1\in B_1^{k-N_\mathcal{O}}} \prod_{j=1}^{k-N_\mathcal{O}} \exp\Big{[} S_{t_0}(V_{\beta^1_j})\Big{]}\\
& = C' \left(  \sum_{b\in B_1}  \exp\Big{[} S_{t_0}(V_{b})\Big{]} \right)^{k-N_\mathcal{O}}\\
&\leq C' \exp\Big{[}(k-N_\mathcal{O})t_0 (\mathcal{P}(1) + 2\varepsilon_0)\Big{]}~~ \text{ thanks to \eqref{rainy}}.
\end{aligned}
\end{equation}

In particular, thanks to (\ref{eq : pression en 1}), we see that for any $M>0$, $\sum_{k=1}^{M |\log h|} \sum_{\beta \in B^k} \|\chi \tilde{U}_\beta
(1-\chi_1)\chi_{t_0} E^0_h\|_{L^2}^2$ is bounded independently of $h$.

\subsection{Regrouping the Lagrangian states}\label{regrouping}
From now on, we fix a compact set $\mathcal{K}\subset X$, and we take $\mathcal{O}\subset \mathcal{K}$.

In \cite[\S 5.1.1]{Ing2}, it is shown that there exists $\varepsilon_\mathcal{K}>0$ such that, if $\mathcal{O}\subset \mathcal{K}$ has a diameter smaller than $\mathcal{\epsilon}_\mathcal{K}$, it is possible to build  an equivalence relation $\sim_\mathcal{O}$ on the set\footnote{The condition that $\beta\in B^k$ is such that $\Phi_\beta(\Lambda_\xi\cap S^*\mathrm{supp}(\chi_{t_0}))\cap S^*\mathcal{O}\neq \emptyset$ implies, in the notations of \cite{Ing2}, that $\beta \in B_{k'}^{\chi_1}$ for some $k'\in \N$, so we can indeed apply the constructions of \cite[\S 5.1.1]{Ing2}.} 
$$\left\{\beta \in \bigcup_{k\in \mathbb{N}} B^k ; \Phi_\beta(\Lambda_\xi\cap S^*\mathrm{supp}(\chi_{t_0}))\cap S^*\mathcal{O}\neq \emptyset \right\}$$
with the following properties:
\begin{enumerate}
\item If $\beta\sim_\mathcal{O} \beta'$, then\footnote{This implication is proven in Lemma 5.3 of \cite{Ing2}. Note that, if the intersection the domains of $\varphi_{\beta,\mathcal{O}}$ and $\varphi_{\beta',\mathcal{O}}$ is non-empty, then the converse implication is true.} for all $x\in \mathcal{O}$ belonging to the domain of definition of $\varphi_{\beta,\mathcal{O}}$ and $\varphi_{\beta',\mathcal{O}}$, we have $\partial \varphi_{\beta,\mathcal{O}} = \partial \varphi_{\beta',\mathcal{O}}$.

\item Let us write $\tilde{\mathcal{B}}^\mathcal{O}:= \Big{(}\bigcup_{k\in \mathbb{N}} B^k\Big{)} \backslash \sim_\mathcal{O}$, and, for each $\tilde{\beta}\in \tilde{\mathcal{B}}^\mathcal{O}$:
\begin{equation}\label{rhodia}
\tilde{n}(\tilde{\beta}):= \left\lceil \frac{1}{t_0} \min \{n\in \mathbb{N};~~ \exists \beta \in B^n \text{ such that } \beta\in \tilde{\beta}\} \right\rceil.
\end{equation}
Then there exists $\mathcal{N}$ such that for each $k\in \N$ and each $\beta\in B^k$, we have
$$\beta \in \tilde{\beta} \Longrightarrow k \leq \tilde{n}(\tilde{\beta}) + \mathcal{N}.$$
This point is precisely \cite[Lemma 5.2]{Ing2}.

\item It is possible for each $\tilde{\beta}\in \tilde{\mathcal{B}}^\mathcal{O}$
 to build a phase function $\varphi_{\tilde{\beta}} : \mathcal{O}\rightarrow \mathbb{R}$ such that 
for every $\beta\in \tilde{\beta}$, 
we have $\partial \varphi_{\tilde{\beta}}(x)=\partial \varphi_{\beta,\mathcal{O}}(x)$ for every $x\in \spt(\varphi_{\beta,\mathcal{O}})$. 

\end{enumerate}

Note that the number of elements in $\{\tilde{\beta}\in\tilde{\mathcal{B}}^{\mathcal{O}}; \tilde{n}(\tilde{\beta})\leq N\}$ grows at most exponentially with $N$, thanks to point 2 above.

It follows from what precedes that there exists a constant $C_\mathcal{O}>0$ such that for all $\tilde{\beta}\neq \tilde{\beta}'\in \tilde{\mathcal{B}}^{\mathcal{O}}$, we have
\begin{equation}\label{hurry7}
|\partial \varphi_{\tilde{\beta}} (x)- \partial \varphi_{\tilde{\beta'}}(x)|\geq C_{\mathcal{O}} e^{- \sqrt{b_0} \max(\tilde{n}(\tilde{\beta}),\tilde{n}(\tilde{\beta'}))}. 
\end{equation}

 Let $\mathcal{O}\subset \mathcal{K}$ has diameter smaller than $\varepsilon_\mathcal{K}$, and let $\chi\in C_c^\infty(\mathcal{O})$.
For every $\tilde{\beta}\in \tilde{\mathcal{B}}^\mathcal{O}$, we set 
\begin{equation*}
a_{\tilde{\beta},\chi}:= \sum_{\beta\in \tilde{\beta}} a_{\beta,\chi} e^{i(\varphi_{\beta,\mathcal{O}}-\varphi_{\tilde{\beta}})/h}.
\end{equation*}
Then $a_{\tilde{\beta},\chi}\in S^{comp}(X)$, since the functions $\varphi_{\beta,\mathcal{O}}-\varphi_{\tilde{\beta}}$ are constant, and, for any $\varepsilon>0$ and $k\in \mathbb{N}$, there exists a $C_{\chi,\epsilon,k}>0$ such that
\begin{equation}\label{sheriff5}
\sum_{\substack{\tilde{\beta}\in \tilde{\mathcal{B}}^\mathcal{O}\\ \tilde{n}(\tilde{\beta})= n}} \|a_{\tilde{\beta},\chi}\|_{C^k}^2 \leq C_{\chi,\varepsilon,k} n^{2k}
e^{n(\mathcal{P}(1)+\varepsilon)}.
\end{equation}

\subsection{End of the proof of Theorem \ref{blacksabbath4}}
Let $0<M< \frac{1}{2\sqrt{b_0}}$, and let $k\in \N$ with $k t_0  \leq M |\log h|$. We have 
\begin{align*}
\chi (\chi_1 \tilde{U}(t_0))^k (1-\chi_1)\chi_{t_0} E^0_h &= \sum_{\tilde{\beta}\in \tilde{\mathcal{B}}^\mathcal{O}} \sum_{\beta\in B^k \cap \tilde{\beta}} \tilde{U}_\beta
(1-\chi_1)\chi_{t_0} E^0_h = \sum_{\substack{\tilde{\beta}\in \tilde{\mathcal{B}}^\mathcal{O} \\ \tilde{n}(\tilde{\beta})\leq k}}  e^{ \frac{i}{h} \varphi_{\tilde{\beta}}} a_{k,\tilde{\beta},\chi},
\end{align*}

where $a_{k,\tilde{\beta},\chi} = \sum_{\beta \in \tilde{\beta}\cap B^k}  a_{\beta,\chi} e^{i(\varphi_{\beta,\mathcal{O}}-\varphi_{\tilde{\beta}})/h}$. Note that, thanks to (\ref{eq:BoundedWithPress}), we have
$$\sum_{\tilde{\beta}\in \tilde{\mathcal{B}}^\mathcal{O}} \|a_{k,\tilde{\beta},\chi}\|_{L^2}^2 \leq C e^{ k t_0 \mathcal{P}(1) + 2 \varepsilon_0}.$$
In particular, by taking $\varepsilon_0$ small enough, we have, for any $\varepsilon>0$, 
$\sum_{\tilde{\beta}\in \tilde{\mathcal{B}}^\mathcal{O}} \|a_{k,\tilde{\beta},\chi}\|_{L^2}^2 \leq C_\varepsilon h^{M \mathcal{P}(1) - \varepsilon}.$

On the other hand, thanks to (\ref{hurry7}), if $\tilde{\beta} \neq \tilde{\beta'}$ satisfy $\tilde{n}(\tilde{\beta}), \tilde{n}(\tilde{\beta'})\leq k$, we have $|\partial \varphi_{\tilde{\beta}} (x)- \partial \varphi_{\tilde{\beta'}}(x)|\geq C_{\mathcal{O}} h^{\frac{1}{2}-\varepsilon}.$ Therefore, just as in the proof of Theorem \ref{th : L2}, non-stationary phase (Proposition \ref{nonstat}) allows us to conclude that
\begin{equation}\label{eq:BorneL2}
\left\|\chi (\chi_1 \tilde{U}(t_0))^k (1-\chi_1)\chi_{t_0} E^0_h\right\|_{L^2}^2 = O_\varepsilon \left( h^{M \mathcal{P}(1) - \varepsilon}\right).
\end{equation}

Now, recall that we took $\chi_1 \equiv 1$ on the support of $\chi$ so that (\ref{canardsauvage4}) holds. However, this assumption was not needed to obtain (\ref{eq:BoundedWithPress}), or in the construction of section \ref{regrouping}. Therefore, (\ref{eq:BorneL2}) holds as soon as $\chi$ has a small enough support, and we deduce that
$\left\|(\chi_1 \tilde{U}(t_0))^k (1-\chi_1)\chi_{t_0} E^0_h\right\|_{L^2}^2 = O_\varepsilon \left( h^{M \mathcal{P}(1)- \varepsilon}\right).$ In particular, using the triangular inequality and the fact that $\|\chi_1 \tilde{U}(t_0)\|_{L^2\rightarrow L^2} \leq 1$, we obtain that, for any $M_r >M$

$$\left\| \sum_{k= M |\log h|}^{M_r |\log h|} \chi (\chi_1 \tilde{U}(t_0))^k
(1-\chi_1)\chi_{t_0} E^0_h \right\|_{L^2} = O_\varepsilon \left(h^{\frac{M \mathcal{P}(1)}{2} - \varepsilon}\right).$$

Therefore, thanks to (\ref{canardsauvage4}), we deduce that
\begin{equation}\label{canardsauvage5}
\begin{aligned}
\chi E_h &=  \sum_{k=1}^{M|\log h| + \mathcal{N}}\sum_{\beta \in B^k} \tilde{U}_\beta
(1-\chi_1)\chi_{t_0} E^0_h + O_\varepsilon \left(h^{\frac{M \mathcal{P}(1)}{2} - \varepsilon}\right)
\\&= \sum_{\substack{\tilde{\beta}\in \tilde{\mathcal{B}}^\mathcal{O}\\
                \tilde{n}(\tilde{\beta})\leq M|\log h|}} a_{\tilde{\beta},\chi} e^{\frac{i}{h} \varphi_{\tilde{\beta}}} + \mathcal{R}_h + O_\varepsilon \left(h^{\frac{M \mathcal{P}(1)}{2} - \varepsilon}\right),
                \end{aligned}
\end{equation}
where the remainder $\mathcal{R}_h$ corresponds to the sum over the $\beta \in B^k$ for some $k\leq M |\log h| + \mathcal{N}$ which do not belong to some $\tilde{\beta}$ with $\tilde{n}(\tilde{\beta})\neq M |\log h|$. By point 2. in section \ref{regrouping}, we must have $k\geq M |\log h|$. Hence, the same argument as the one leading to (\ref{eq:BorneL2}) shows that $\|\mathcal{R}_h \|_{L^2} =O_\varepsilon \left(h^{\frac{M \mathcal{P}(1)}{2} - \varepsilon}\right)$, which concludes the proof.

\subsection{Proof of Proposition \ref{prop: small scale} and Corollary \ref{cor:borne Lp}}\label{sec: proof Lp}
\begin{proof}[Proof of Proposition \ref{prop: small scale}]
Recall that we take $r_h:= h^\alpha$, with $\alpha< \left[2\left( 1+ \frac{\sqrt{b_0} d}{|\mathcal{P}(1)|}\right)\right]^{-1}$. 

Take $M^{\varepsilon_0}:= \frac{d\alpha}{|\mathcal{P}(1)|} + \varepsilon_0$, with $\varepsilon_0$ small enough so that $M^{\varepsilon_0}< \frac{1}{2 \sqrt{b_0}}$.

Let $x_0\in X$. By Theorem \ref{blacksabbath4}, in a neighborhood of $x_0$, we may write $E_h$ as
$$E_h = S_h + R_h,$$
with $\|R_h\|_{L^2(X)} = O(h^{\frac{M^{\varepsilon_0}}{2}|\mathcal{P}(1)|- \varepsilon}) = o(r_h^{d/2})$ provided we take $\varepsilon< \frac{\varepsilon_0 |\mathcal{P}(1)|}{2}$,
and $S_h$ is a sum of Lagrangian states of the form
$$ S_h = \sum_{\substack{\tilde{\beta}\in \tilde{\mathcal{B}}\\
                \tilde{n}(\tilde{\beta})\leq M^{\varepsilon_0}|\log h|}} e^{i \varphi_{\tilde{\beta}}(x)/h}
a_{\tilde{\beta}}(x;h),
$$
with
\begin{equation}\label{sheriff4}
\sum_{\substack{\tilde{\beta}\in \tilde{\mathcal{B}}\\ \tilde{n}(\tilde{\beta})=n}} \|a_{\tilde{\beta}}\|_{L^2}^2 \leq C_{\varepsilon}
e^{n(\mathcal{P}(1)+\varepsilon)}.
\end{equation}
and there exists a constant $C>0$ such that for all $\tilde{\beta}\neq  \tilde{\beta}'\in \tilde{\mathcal{B}}$ with $\tilde{n}(\tilde{\beta})\leq M^{\varepsilon_0}|\log h|$, we have
\begin{equation}\label{hurry9}
|\partial \varphi_{\tilde{\beta}} (x)- \partial \varphi_{\tilde{\beta'}}(x)|\geq C h^{\sqrt{b_0}M^{\varepsilon_0}}. 
\end{equation}

We have 
\begin{equation*}
\begin{aligned}
\int_{B(x_0,r_h)} |E_h|^2 &\leq \int_{B(x_0,r_h)} |S_h|^2 +\int_{B(x_0,r_h)}( |R_h|^2 + 2 |S_h R_h|).
\end{aligned}
\end{equation*}
Therefore, if we can show that $\int_{B(x_0,r_h)} |S_h|^2 = O(r_h^d)$, then we will have that $\int_{B(x_0,r_h)} |E_h|^2 = O(r_h^d)$.

Let $\chi_h\in C_c^\infty(X;[0,1])$ be supported in $B(x_0,2r_h)$, and equal to one in $B(x_0,r_h)$, so that $\int_{B(x_0,r_h)} |S_h|^2\leq \int_X \chi_h |S_h|^2$.
We have 
\begin{equation}\label{eq: decompo integrale}
\begin{aligned}
\int_X \chi_h |S_h|^2 &= \sum_{\substack{\tilde{\beta}\in \tilde{\mathcal{B}}\\
                \tilde{n}(\tilde{\beta})\leq \tilde{M}^{\varepsilon}|\log h|}} \int_X \chi_h
|a_{\tilde{\beta}}(x;h)|^2 \\
&+ \sum_{\substack{\tilde{\beta}\neq \tilde{\beta}'\in \tilde{\mathcal{B}}\\
                \tilde{n}(\tilde{\beta}),\tilde{n}(\tilde{\beta}')\leq \tilde{M}^{\varepsilon}|\log h|}} \int_X \chi_h e^{i \varphi_{\tilde{\beta},\tilde{\beta}'}(x)/h}
a_{\tilde{\beta},\tilde{\beta}'}(x;h),
\end{aligned}
\end{equation}
where $\varphi_{\tilde{\beta},\tilde{\beta}'} = \varphi_{\tilde{\beta}} - \varphi_{\tilde{\beta}'}$ and $a_{\tilde{\beta},\tilde{\beta}'} = a_{\tilde{\beta}}\overline{a_{\tilde{\beta}'}}$.

From (\ref{eq: borne C^0 itere}), we see that 
\begin{equation*}
\int_X \chi_h |a_{\beta,\chi}|^2 \leq C r_h^d \exp\Big{[}\sum_{i=0}^{n_1(\beta)} S_{t_0}(V_{\beta^1_i})\Big{]},
\end{equation*}
so that by (\ref{eq : pression en 1}) and (\ref{rainy}), we see that the first sum in (\ref{eq: decompo integrale}) is a $O(r_h^d)$.

Let us consider the terms of the other sum in (\ref{eq: decompo integrale}). By a change of variables, they can be put in the form
$$ \int_{X} \chi_h a_{\beta,\beta'}(x;h) e^{i\varphi_{\beta,\beta'}(x)/h} \mathrm{d}x = r_h^d\int_{B(0,2)} \chi(x) a_{\beta,\beta'} \left(\exp_{x_0}(r_h x);h\right) e^{i \varphi_{\beta,\beta'}\left(\exp_{x_0}(r_h x)\right)/h} \mathrm{d}x,$$
with $\chi$ independent of $h$, supported in a ball of radius 2. By (\ref{hurry9}), the gradient of the phase is bounded from below by
$$r_h \times C h^{\sqrt{b_0}M^{\varepsilon_0}} = C h^{\alpha\left(1+ \frac{d\sqrt{b_0}}{|\mathcal{P}(1)|}\right) + \sqrt{b_0}\varepsilon_0},$$
so that, thanks to (\ref{eq:DefAlpha1}), for $\varepsilon_0$ small enough, we may apply Proposition \ref{nonstat}.  We deduce that this integral is $O(h^\infty)$. Since the number of terms in the second sum in (\ref{eq: decompo integrale}) is bounded by a power of $h$, we conclude that the second term in (\ref{eq: decompo integrale}) is a $O(h^\infty)$.

This proves the claim.
\end{proof}

Let us now prove Corollary \ref{cor:borne Lp}.
\begin{proof}[Proof of Corollary \ref{cor:borne Lp}]

Take $\alpha$ satisfying (\ref{eq:DefAlpha1}), and $r_h = h^\alpha$.
Let $x_0\in X$. We shall denote by $\exp_{x_0} : T_{x_0} X \simeq \R^d\rightarrow X$ the exponential map centered at $x_0$. Let $\psi \in C_c^\infty( \R^+; [0,1])$ be equal to one on $[0,1/2]$ and vanish on $[1,\infty)$. For $y\in \R^d$, we set 
$$\tilde{E}_h (y ; x_0) = \psi(|y|_{x_0}) E_h ( \exp_{x_0} (r_h y)).$$

By Proposition \ref{prop: small scale},  we have 
$$\|\tilde{E}_h (\cdot; x_0)\|^2_{L^2(\R^d)} \leq \int_{B(0,1)} |E_h ( \exp_{x_0} (r_h y))|^2 \mathrm{d}y \leq \frac{C}{r_h^d} \int_{B(x_0, r_h)} |E_h(x)|^2 \mathrm{d}x =O_{h\rightarrow 0} (1).$$

We define the operator $Q_h$ on $T_{x_0}^*X\simeq \R^d$.

$$Q_h := -\Big{(}\frac{h}{r_h}\Big{)}^2 \psi(|y|_{x_0}/10) \Big{(} \sum_{i,j} g^{i,j}(r_h y) \frac{\partial^2}{\partial y_i \partial y_j} + \frac{1}{D_g(r_h y)} \frac{\partial}{\partial y_i} \big{(} (D_g g^{i,j}) (r_h y) \big{)} \frac{\partial}{\partial y_j }\Big{)},$$
where $D_g := \sqrt{\det (g_{i,j})}$, and where $g^{i,j}$ are the coefficients of the metric in the coordinates $y= \exp_{x_0}^{-1}(x)$.

As shown in \cite[\S 3.4]{hezari2016quantum}, we have
$$(Q_h-1) \tilde{E}_h = O\Big{(} \frac{h}{r_h}\Big{)} \|\tilde{E}_h\|_{L^2}.$$

Therefore, we can apply \cite[Theorem 7.12]{Zworski_2012} to obtain
$$\|\tilde{E}_h\|_{L^\infty} \leq C \Big{(}\frac{h}{r_h}\Big{)}^{-(d-1)/2},$$
and (\ref{eq: borne Linfini}) follows.

Let $$p_d:= \frac{2(d+1)}{d-1}.$$
We may also apply \cite[Theorem 10.10]{Zworski_2012} to obtain 
$$\|\tilde{E}_h\|_{L^{p_d}} \leq C \Big{(}\frac{h}{r_h}\Big{)}^{-1/p_d}.$$

Now, by a change of variables, there exists $C>0$ such that
$$\int_{B(x_0, r_h)} |E_h|^{p_d} \leq C r_h^d \int_{B(0,10)} |\tilde{E}_h|^{p_d}(y) \mathrm{d}y.$$ 

We may then cover the support of $\chi$ by geodesic balls of radius $r_h$. Since the number of such balls is a $O(r_h^{-d})$, we obtain

$$\int_X \chi|E_h|^{p_d}\leq C \frac{r_h}{h},$$
which gives us the result.
\end{proof}
\begin{appendices}
  \renewcommand\thetable{\thesection\arabic{table}}
  \renewcommand\thefigure{\thesection\arabic{figure}}
  \section{Review of semiclassical analysis} \label{pivoine}
\subsection{Pseudodifferential calculus} \label{greve}

Let $X$ be a Riemannian manifold.
We will say that a function $a(x,\xi;h)\in C^{\infty}(T^*X\times (0,1])$ is in the class $S^{comp}(T^*X)$ if its support and of all its semi-norms are bounded independently of $h$.

\begin{definition}\label{defsymbclassique}
Let $a\in S^{comp}(T^*X)$. We will say that $a$ is a \emph{classical symbol} if there exists a sequence of symbols $a_k\in C^\infty(T^*X)$, independent of $h$, such that for any $n\in \mathbb{N}$, 
$$a-\sum_{k=0}^n h^k a_k \in h^{n+1} S^{comp}(T^*X).$$
We will then write $a^0(x,\xi):= \lim\limits_{h\rightarrow 0} a(x,\xi;h)$ for the \emph{principal symbol} of $a$.

We will sometimes write that $a\in S^{comp}(X)$ if it can be written as $a(x;h)= \tilde{a}_h(x)$,
where the functions $\tilde{a}_h\in C_c^\infty(X)$ have all their semi-norms and supports bounded independently of $h$.
\end{definition}

We associate to $S^{comp}(T^*X)$ the class of pseudodifferential operators
$\Psi_h^{comp}(X)$, through a surjective quantization map
\begin{equation*}Op_h:S^{comp}(T^*X)\longrightarrow \Psi^{comp}_h(X).
\end{equation*} This quantization
map is defined using coordinate charts, and the standard quantization
on $\mathbb{R}^d$. It is therefore not intrinsic, but another choice of quantization would only affect the resulting operator by a term $O_{L^2\rightarrow L^2}(h)$.

Let us detail this construction a bit more. Suppose that $\mathcal{U}\subset X$ is an open set small enough so that we can define a coordinate chart, i.e. a diffeomorphism $\gamma : \mathcal{U}\longrightarrow \mathcal{U}_0 \subset \R^d$. For each $x\in \mathcal{U}$, we denote its co-Jacobian matrix by $\partial \gamma (x)^T : T^*_x X \longrightarrow \R^d$. Let $\chi \in C_c^\infty( \mathcal{U})$. If $a\in S^{comp}(T^*X)$ and $f\in L^2(X)$, we define
\begin{equation}\label{eq:DefQuantification}
(Op_h(a) \chi f)(x):= \frac{1}{(2\pi h)^d} \int_X \mathrm{d}y \int_{T_x^* X} \mathrm{d}\xi e^{\frac{i}{h} (\gamma(x)-\gamma(y))\cdot (\partial \gamma (x)^T) (\xi) } a(x, \xi )  \chi(x) \chi(y) f(y) g_\gamma(y,\xi) \mathrm{d}y \mathrm{d}\xi,
\end{equation}
where $g$ is the smooth function such that $g_\gamma(y,\xi) \mathrm{d}y \mathrm{d}\xi$ is the pushforward of the Liouville measure by $ \R^d\times \R^d \ni (y,v)\mapsto \left(\gamma(y), (\partial \gamma(x)^T)^{-1}(v)\right)$.

The quantization $Op_h$ is then built by taking a locally finite partition of unity associated to sets $\mathcal{U}_i$ where coordinate charts are defined, and by summing terms of the form \eqref{eq:DefQuantification}.

For more details on this quantization procedure, we refer the reader
to \cite[Chapter 14]{Zworski_2012}.

For $a\in S^{comp}(T^*X)$, we say that its essential support is equal to a given
compact $K\Subset T^*X$,
\begin{equation*} \text{ ess supp}_h a = K \Subset T^*X,
\end{equation*}
if and only if, for all $\chi \in S(T^*X)$,
\begin{equation*}\spt \chi \subset (T^*X\backslash K) \Rightarrow \chi a \in h^\infty S(T^*
X).
\end{equation*}
For $A\in \Psi^{comp}_h(X), A=Op_h(a)$, we define the wave front set of $A$ as:
\begin{equation*}WF_h(A)= \text{ ess supp}_h a,
\end{equation*}
noting that this definition does not depend on the choice of the
quantization. When $K$ is a compact subset of $T^*X$ and $WF_h(A)\subset K$, we will sometimes say that $A$ is \emph{microsupported} inside $K$.

If $U,V$ are bounded open subsets of $T^*X$, and if $T,T' : L^2(X)\rightarrow L^2(X)$ are bounded operators, we shall say that $T\equiv T'$ \emph{microlocally} near $U\times V$ if there exist bounded open sets $\tilde{U}\supset \overline{U}$ and $\tilde{V} \supset \overline{V}$ such that for any $A,B\in \Psi_h^{comp}(X)$ with $WF_h(A)\subset \tilde{U}$ and $WF_h(B)\subset \tilde{V}$, we have
\begin{equation*}
A(T-T')B = O_{L^2\rightarrow L^2} (h^\infty).
\end{equation*}

\subsubsection*{Tempered distributions}
Let $u=(u(h))$ be an $h$-dependent family of distributions in $\mathcal{D}'(X)$. We say it is \emph{$h$-tempered} if for any bounded open set $U\subset X$, there exists $C\bel 0$ and $N\in \mathbb{N}$ such that
\begin{equation*}
\|u(h)\|_{H_h^{-N}(U)}\leq C h^{-N},
\end{equation*}
where $\|\cdot\|_{H_h^{-N}(U)}$ is the semiclassical Sobolev norm.

For a tempered distribution $u=(u(h))$, we say that a point $\rho\in T^*X$ does not lie in the wave front set $WF_h(u)$ if there exists a neighborhood $V$ of $\rho$ in $T^*X$ such that for any $A\in \Psi_h^{comp}(X)$ with $WF_h(A)\subset V$, we have $Au=O(h^\infty)$.

Let us now recall Egorov's theorem (see \cite[Theorem 11.1]{Zworski_2012} or \cite[Proposition 3.8]{DG}). Recall that $U(t)$ is the semiclassical Schrödinger propagator $U(t)= e^{it h \Delta}$.

\begin{proposition}[Egorov's theorem]\label{prop:Egorov}
Let $A\in \Psi_h^{comp}(X)$ and let $t\in \R$. There exists $A_t\in \Psi_h^{comp}(X)$ such that $U(-t) A U(t) = A_t + O_{L^2\to L^2} (h^\infty)$. Furthermore, $WF_h(A_t) = \Phi^t (WF_h(A))$.

In particular, if $t\in \R$ and $A,B\in \Psi_h^{comp}(X)$ are such that $\Phi^t( WF_h(A))\cap WF_h(B)=\emptyset$, then we have
\begin{equation*}
A U(t) B= O_{L^2\rightarrow L^2}(h^\infty).
\end{equation*}
\end{proposition}

\subsubsection*{Non-stationary phase}
Let $a,\phi\in S^{comp}(X)$, with $\spt(a)\subset \spt(\phi)$ and $\phi$ real-valued.
We consider the oscillatory integral:
$$I_h(a,\phi):= \int_{X} a(x) e^{\frac{i\phi(x,h)}{h}} \mathrm{d}x.$$

The following result is classical, and its proof similar to that of \cite[Lemma 3.12]{Zworski_2012}.
\begin{proposition}[Non stationary phase]\label{nonstat}
Let $\epsilon>0$. Suppose that there exists $C>0$ such that, $\forall x\in \spt(a), \forall  0<h<h_0$, $|\partial \phi(x,h)|\geq C h^{1/2-\epsilon}$. Then
$$I_h(a,\phi)=O(h^\infty).$$
\end{proposition}
\subsection{Lagrangian distributions and Fourier Integral Operators}\label{DSK}
Let us now recall some facts about Lagrangian distributions and Fourier Integral Operators, following \cite[\S 3.2]{DG}.
\subsubsection*{Phase functions}
Let $\phi(x,\theta)$ be a smooth real-valued function on some open subset $U_\phi$ of $X\times \mathbb{R}^L$, for some $L\in \mathbb{N}$. We call $x$ the \emph{base variables} and $\theta$ the \emph{oscillatory variables}. We say that $\phi$ is a \emph{nondegenerate phase function} if the differentials $\mathrm{d}(\partial_{\theta_1} \phi),..., \mathrm{d}(\partial_{\theta_L}\phi)$ are linearly independent on the \emph{critical set }
\begin{equation*}
C_\phi:=\{ (x,\theta); \partial_\theta \phi =0 \} \subset U_\phi.
\end{equation*}
In this case
\begin{equation*}
\Lambda_\phi:= \{(x,\partial_x \phi(x,\theta)); (x,\theta)\in C_\phi \} \subset T^*X
\end{equation*}
is an immersed Lagrangian manifold. By shrinking the domain of $\phi$, we can make it an embedded Lagrangian manifold. We say that $\phi$ \emph{generates} $\Lambda_\phi$.

\subsubsection*{Lagrangian distributions}
Given a phase function $\phi$ and a symbol $a\in S^{comp}(U_\phi)$, consider the $h$-dependent family of functions
\begin{equation}\label{massai}
u(x;h)= h^{-L/2} \int_{\mathbb{R}^L} e^{i\phi(x,\theta)/h} a(x,\theta;h) \mathrm{d}\theta.
\end{equation}
We call $u=(u(h))$ a \emph{Lagrangian distribution}, (or a \emph{Lagrangian state}) generated by $\phi$. By the method of non-stationary phase, if $\spt  (a)$ is contained in some $h$-independent compact set $K\subset U_\phi$, then
\begin{equation*}
WF_h(u)\subset \{(x,\partial_x \phi(x,\theta)); (x,\theta)\in C_\phi\cap K\}\subset \Lambda_\phi.
\end{equation*}

\begin{definition}\label{Grenoble}
Let $\Lambda\subset T^*X$ be an embedded Lagrangian submanifold. We say that an $h$-dependent family of functions $u(x;h)\in C_c^\infty(X)$ is a (compactly supported and compactly microlocalized) \emph{Lagrangian distribution associated to $\Lambda$}, if it can be written as a sum of finitely many functions of the form (\ref{massai}), for different phase functions $\phi$ parametrizing open subsets of $\Lambda$, plus an $O(h^\infty)$ remainder. We will denote by $I^{comp}(\Lambda)$ the space of all such functions.
\end{definition}

\subsubsection*{Fourier integral operators}
Let $X, X'$ be two manifolds of the same dimension $d$, and let $\kappa$ be a symplectomorphism from an open subset of $T^*X$ to an open subset of $T^*X'$. Consider the Lagrangian
\begin{equation*}
\Lambda_\kappa =\{(x,\nu;x',-\nu'); \kappa(x,\nu)=(x',\nu')\}\subset T^*X\times T^*X'= T^*(X\times X').
\end{equation*}
A compactly supported operator $U:\mathcal{D}'(X)\rightarrow C_c^\infty(X')$ is called a (semiclassical) \emph{Fourier integral operator} associated to $\kappa$ if it can be written as an operator which is a $O_{\mathcal{D}'(X)\rightarrow C_c^\infty(X')}(h^\infty)$ plus an operator whose Schwartz kernel $K_U(x,x')$ lies in $h^{-d/2}I^{comp}(\Lambda_\kappa)$. We write $U\in I^{comp}(\kappa)$. The $h^{-d/2}$ factor is explained as follows: the normalization for Lagrangian distributions is chosen so that $\|u\|_{L^2}\sim 1$, while the normalization for Fourier integral operators is chosen so that $\|U\|_{L^2(X)\rightarrow L^2(X')} \sim 1$.

Note that if $\kappa\circ \kappa'$ is well defined, and if $U\in I^{comp}(\kappa)$ and $U'\in I^{comp}(\kappa')$, then $U\circ U'\in I^{comp} (\kappa\circ \kappa')$.

If $U\in I^{comp}(\kappa)$ and $O\subset T^*X$ is an open bounded set, we shall say that $U$ is \emph{microlocally unitary }near $O$ if $U^* U \equiv I_{L^2(X)\rightarrow L^2(X)}$ microlocally near $O\times \kappa (O)$.

\subsection{Iterations of Fourier integral operators}\label{bismark}
We recall here the main results from \cite[\S 4]{NZ} concerning the
iterations of semiclassical Fourier integral operators in $T^*\mathbb{R}^d$.

Let $V\subset T^* \mathbb{R}^d$ be an open neighborhood of $0$, and take
a sequence of symplectomorphisms $(\kappa_i)_{i=1,...,N}$ from $V$ to
$T^*\mathbb{R}^d$, such that $\forall i \in \{1,...,N\}$, we have
$\kappa_i(0)\in V$, and the following projection:
\begin{equation*}(x_1,\xi_1 ; x_0, \xi_0) \mapsto (x_1,\xi_0) \text{    where    }
(x_1,\xi_1)= \kappa (x_0, \xi_0 ) 
\end{equation*}
is a diffeomorphism close to the origin.
We consider Fourier integral operators $(T_i)$ associated to $\kappa_i$
and which are
microlocally unitary near an open set $U\times U$, where $U \Subset V$ which
contains the origin. Let $\Omega\subset \mathbb{R}^d$ be an open set such
that $U\Subset T^*\Omega$, and, for all $i$, $\kappa_i(U) \Subset T^*\Omega$.
For each $i$, we take a smooth cut-off function $\chi_i\in C_c^\infty (U ;
[0,1] )$, and let
\begin{equation} \label{blabla}
S_i := Op_h(\chi_i) \circ T_i.
\end{equation}
Let us consider a family of Lagrangian manifolds $\Lambda_k = \{
(x,\phi_k'(x)) ; x\in \Omega\} \subset T^*\mathbb{R}^d, ~~ k=0,...,N$ such
that:
\begin{equation} \label{leffe}
|\partial^\alpha \phi_k|\leq C_\alpha,~~~~0\leq k\leq N~~ \alpha\in
\mathbb{N}^d.
\end{equation}
We assume that there exists a sequence of integers $(i_k\in
\{1,...,J\})_{k=1,...,N}$ such that
\begin{equation}\label{eq:Caseprojettebien}
\kappa_{i_{k+1}}(\Lambda_k\cap U) \subset \Lambda_{k+1},~~k=0,...,N-1.
\end{equation}

We define $g_k$ by
\begin{equation*}g_k(x) = \pi\circ \kappa_{i_k}^{-1}(x,\phi_k'(x)).
\end{equation*}
That is to say,
$\kappa_{i_k}^{-1}(x,\phi_k'(x))=(g_k(x),\phi_{k-1}'(g_k(x)))$.

We will say that a point $x\in \Omega$ is $N$-admissible if we can define
recursively a sequence by $x^N=x$, and, for $k=N,...,1$,
$x^{k-1}=g_k(x^k)$. This procedure is possible if, for any $k$, $x^k$ is
in the domain of definition of $g_k$.

Let us assume that, for any admissible sequence $(x^N, ...,x^0)$, the
Jacobian matrices are uniformly bounded from above:
\begin{equation*}\Big{\|}\frac{\partial x^k}{\partial x^l}\Big{\|}=\Big{\|}\frac{\partial
(g_{k+1}\circ g_{k+2}\circ ...\circ g_l)}{\partial x^l}(x^l)\Big{\|}\leq
C_D,~~~0\leq k<l\leq N,
\end{equation*}
where $C_D$ is independent on $N$. This assumption roughly says that the
maps $g_k$ are (weakly) contracting.

We will also use the notation
\begin{equation*}D_k:= \sup\limits_{x\in \Omega} |\det
dg_k(x)|^{1/2},~~~J_k:=\prod_{k'=1}^k D_{k'}, 
\end{equation*}
and assume that the $D_k$'s are uniformly bounded: $1/C_D\leq D_k\leq C_D$.

The following result can be found in \cite[Proposition 4.1]{NZ} (see also \cite[Remark 4.1]{NZ} for the $C^\ell$ bounds on the remainders).

\begin{proposition} \label{haroun}
We use the above definitions and assumptions, and take $N$ arbitrarily large,
possibly varying with $h$.
Take any $a\in C_c^\infty(\Omega)$ and consider the Lagrangian state $u=a
e^{i\phi_0/h}$ associated with the Lagrangian $\Lambda_0$. Then we may
write:
\begin{equation*} (S_{i_N}\circ...\circ S_{i_1}) (a e^{i\phi_0/h})(x) = e^{i\phi_N(x)/h}
\big{(} \sum_{j=0}^{L-1} h^j a_j^N(x) + h^L R_L^N(x,h)\big{)},
\end{equation*}
where each $a_j^N\in C_c^\infty(\Omega)$ depends on $h$ only through $N$, and
$R_L^N\in C^\infty ((0,1]_h,\mathcal{S}(\mathbb{R}^d))$.
If $x^N\in \Omega$ is $N$-admissible, and defines a sequence $(x^k),
k=N,...,1$, then
\begin{equation*}|a_0^N(x^N)| = \Big{(}\prod_{k=1}^N \chi_{i_k} (x^k,\phi_k' (x^k))|\det
dg_k(x^k)|^{\frac{1}{2}}\Big{)} |a(x^0)|,
\end{equation*}
otherwise $a^N_j(x^N)=0, ~~ j=0,...,L-1$.
We also have the bounds
\begin{equation}\label{republique}
\|a_j^N\|_{C^\ell(\Omega)}\leq C_{j,\ell} J_N(N+1)^{\ell+3j}
\|a\|_{C^{\ell+2j}(\Omega)},~~~j=0,...,L-1, \ell\in \mathbb{N},
\end{equation}
\begin{equation}\label{coligny}\|R_L^N\|_{L^2(\mathbb{R}^d)}\leq C_L
\|a\|_{C^{2L+d}(\Omega)}(1+C_0h)^N\sum_{k=1}^N J_k k^{3L+d},
\end{equation}
\begin{equation}\label{antigna}
\|R_L^N\|_{C^\ell(\mathbb{R}^d)}\leq C_{L,l} h^{-d/2-\ell}
\|a\|_{C^{2L+d}(\Omega)}(1+C_0h)^N\sum_{k=1}^N J_k k^{3L+d}.
\end{equation}

The constants $C_{j,\ell},C_0$ and $C_L$ depend on the constants in
(\ref{leffe}) and on the operators $\{S_j\}_{j=1}^J$.
\end{proposition}

We shall be using this proposition in the case where for all $k$, we have $D_k\leq \nu < 1$. In this case, the estimates (\ref{republique}), (\ref{coligny}) and (\ref{antigna}) imply that for any $\ell\in \mathbb{N}$, there exists $C_{\ell}$ independent of $N$ such that for any $N\in \mathbb{N}$, we have
\begin{equation}\label{highway78}
\|a^N\|_{C^\ell}\leq \|a^N_0\|_{C^\ell} \big{(} 1+ C_\ell h \big{)}.
\end{equation}

\end{appendices}

\bibliographystyle{alpha}
\bibliography{references}

\providecommand{\bysame}{\leavevmode\hbox to3em{\hrulefill}\thinspace}
\providecommand{\MR}{\relax\ifhmode\unskip\space\fi MR }
\providecommand{\MRhref}[2]{%
  \href{http://www.ams.org/mathscinet-getitem?mr=#1}{#2}
}
\providecommand{\href}[2]{#2}

\end{document}